\documentclass[onefignum,onetabnum]{siamart171218}

\usepackage{lipsum}
\usepackage{amsfonts}
\usepackage{graphicx}
\usepackage{epstopdf}
\usepackage{algorithmic}
\usepackage{amstext,amssymb}
\usepackage{float} 
\usepackage{amsopn}
\usepackage{array} 
\usepackage[caption=false]{subfig}

\ifpdf
  \DeclareGraphicsExtensions{.eps,.pdf,.png,.jpg}
\else
  \DeclareGraphicsExtensions{.eps}
\fi

\newsiamremark{remark}{Remark}
\newsiamremark{hypothesis}{Hypothesis}
\crefname{hypothesis}{Hypothesis}{Hypotheses}
\newsiamthm{claim}{Claim}
\newsiamremark{example}{Example}

\headers{A Cut Finite Element Approach to Solving Inverse Obstacle
  Problems}{E. Burman, C. He and M. Larson}

\title{Comparison of Shape Derivatives using CutFEM for Ill-posed Bernoulli Free Boundary Problem
  \thanks{Submitted to the editors of Journal of Scientific Computing.
\funding{EB and CH were was funded by the EPSRC grant EP/P01576X/1. ML
was funded by The 
Swedish Foundation for Strategic Research Grant No.\ AM13-0029, the Swedish Research Council Grants No. 2017-03911 and the Swedish Research Program Essence}
}}

\author{Erik Burman
	\thanks{Department of Mathematics, 
	University College London, Gower Street, London, UK--WC1E  6BT, United Kingdom 
  	(\email{e.burman@ucl.ac.uk})}
\and 
	Cuiyu He
	\thanks{Department of Mathematics, University College London, Gower Street, London, 
	UK--WC1E  6BT, United Kingdom
  	(\email{c.he@ucl.ac.uk})}	
\and
	Mats G. Larson
	\thanks{Department of Mathematics and Mathematical Statistics, Ume\aa \hspace{1mm}University,
	SE-90187 Ume\aa, Sweden (\email{mats.larson@umu.se})}
}

\newcommand{\bfn}{\boldsymbol n}

\newcommand{\bfv}{\boldsymbol v}

\newcommand{\bn}{\boldsymbol n}

\newcommand{\bfbeta}{\boldsymbol \beta}
\newcommand{\bftheta}{\boldsymbol \theta}

\numberwithin{equation}{section}
\newtheorem{lem}{Lemma}[section]

\newtheorem{rem}{Remark}[section]

\newcommand{\jump}[1]{[\![#1]\!]}

\def\cT{{\mathcal T}}
\def\cE{{\mathcal E}}
\newcommand{\divvr}{ \nabla \cdot}

\def\O{{\Omega}}
\def\o{{\omega}}

\usepackage{todonotes}
\ifpdf
\hypersetup{
  pdftitle={Inverse Obstacle Problem},
  pdfauthor={Erik Burman and Cuiyu He}
}
\fi

\begin{document}

\maketitle

\begin{abstract}
In this paper we discuss a level set approach for the identification of
an unknown boundary in a computational domain. The problem takes the
form of a Bernoulli problem where only the Dirichlet datum is known on the
boundary that is to be identified, but additional information on the Neumann condition is available on the known part of the boundary. The approach uses a classical
constrained optimization problem, where a cost functional is minimized with
respect to the unknown boundary, the position of which is defined implicitly by a level set function. To solve the optimization problem a
steepest descent algorithm using shape derivatives is applied. In each iteration the
cut finite element method is used to obtain high accuracy
approximations of the pde-model constraint for a given level set configuration without re-meshing. We
consider three different shape derivatives. First the classical one,
derived using the continuous optimization problem (optimize then
discretize). Then the functional is first discretized using the CutFEM
method and the shape derivative is evaluated on the finite element functional (discretize then optimize). Finally we consider
a third approach, also using a discretized functional. In this case
we do not perturb the domain, but consider a so-called boundary value
correction method, where a small correction to the boundary position
may be included in the weak boundary condition. Using this correction
the shape derivative may be obtained by
perturbing a distance parameter in the discrete variational
formulation. The theoretical discussion is illustrated with a series
of numerical examples showing that all three approaches produce
similar result on the proposed Bernoulli problem.

\end{abstract}



\section{Introduction}
This paper deals with the reconstruction of the free surface of the
ill-posed free boundary Bernoulli problem. Comparing to the classical free boundary Bernoulli problem, this paper
studies the free boundary problems for which Dirichlet data is known on the
free boundary and Cauchy data is known on the fixed
boundary. Such problems are found for instance in models where perfectly insulated obstacles \cite{afraites2007detecting} need to be detected from data.
Following \cite{BEHLL17} we use the cut finite element method (CutFEM) together  with a level set approach in order to numerically identify the free boundary using the shape optimization method. The level set method is a commonly used tool for inverse problems and optimal design \cite{peng1999pde,osher2001level,Bur01,wang2003level,AJT02, AJT04,BCT05, burger2005survey}. When the level set method is used in the framework of shape optimization or identification, the shape gradient (or steepest descent direction) is obtained by solving partial differential equations in the domain defined by the level set. It is then advantageous to use a fictitious domain type approximation method, provided a sufficient accuracy can be ensured. This is the rationale for combining the CutFEM with level set based optimization. The CutFEM features the following advantages: (1) there is no need to modify the classical basic functions; (2) the approximation has optimal accuracy in the bulk and on the boundary; and (3) it can easily be used in combination with the \textit{level set method}. 
 It has indeed been applied in combination with the level set approach to various shape or topology optimization problems, for instance in \cite{VM17, burman2018shape, bernland2018acoustic, BEHLL19}.

For the shape optimization method, shape sensitivity analysis plays a paramount role. The objective of the present work is to explore the effect of using different shape derivatives in the shape identification problem described above.
First we recall the classical shape derivative obtained by computing the gradient of the Lagrangian functional on the continuous level in an optimize-then-discretize approach.
The gradient is then approximated using the cut finite element method.
We note here that using the classical optimize-then-discrete approach, the shape derivative has two equivalent forms by the structure theorem of Hadamard-Zol\'{e}sio \cite{hadamard1908memoire,delfour2011shapes}, i.e., the domain and boundary representations.  Assuming enough regularity on the continuous level those two forms are equivalent. The applicability of the domain representation is in principle wider, since it requires lower regularity. Moreover, it has been proven to possess certain super-convergence properties compared to the boundary formulation \cite{hiptmair2015comparison,hiptmair2015shape,laurain2016distributed}. In this work, we obtain the domain form for the optimize-then-discretize approach.

One may argue that the discretization of the gradient obtained from the continuous approach only gives an approximate gradient, whose accuracy depends on the mesh-size and that this may prohibit convergence to the minimizer on a fixed mesh. 
In this paper we therefore aim to derive and study shape derivatives for the CutFEM framework using the discretize-then-optimize approach. The advantage is that the shape derivative obtained by this approach in principle can be exact on the mesh-scale considered. However, since this approach optimizes the discretized system directly, the shape derivative may need more terms for the representation. Another 
potential problem with this approach is that although the shape derivative is computed using the discrete system, the descent direction in general is not a function in the finite element space and therefore it still needs to be approximated.

Instead of using the complex formula resulting from the discretize-then-optimize approach, it turns out that we can approximate the shape derivative of the discrete formulation in a much more simpler way. The shape derivative of the discrete system may be obtained through the CutFEM method together with a boundary value correction method \cite{BDT72,BHL18,main2018shifted,cheung2019optimally, Burman2019}. Such a boundary value correction type shape derivative, is also exact for the discrete formulation. 
The derivative only depends on the boundary terms in the Nitsche, or Lagrange muliplier formulation, which could make it possible to tackle more sophisticated problems whose classical shape derivative is difficult to find. The rigorous justification of this boundary value correction shape derivative will be left for future work, instead we will compare its performance numerically with the two other approaches.

To verify and compare the performance of the three different types of derivatives, i.e., the continuous, the discrete and the boundary value correction type, some numerical experiments are presented at the end of this manuscript. Since the objective was to compare the shape derivatives we only consider a simple  steepest descent algorithm for the optimization and it is expected that convergence can be enhanced by applying a more sophisticated method such as the Levenberg-Marquard method proposed in \cite{Burg04}. It turns out that all three shape derivatives have similar performance. 

For another level set based identification method not relying on shape derivatives we refer to \cite{BD10,BD14}.

The paper is organized as follows. In \cref{sec:model problem}, we introduce the model problem. 
Then we introduce the CutFEM for the numerical approximation of the primal and dual solutions in \cref{sec:cutfem}. The various shape derivatives are introduced in \cref{sec:shape-derivatives}. The final optimization algorithm is provided in 
\cref{sec:Optimization-Algorithms}. Finally, the results for numerical experiments are presented in \cref{sec:Numerical Experiments}.

\section{Model problem}\label{sec:model problem}
Let $\hat \O \subset \mathbb{R}^2$ be a simply connected fixed domain and $\Gamma_{f}:= \partial \hat \Omega$. Let
$\mathcal{O}$ be a family of bounded connected domains $\O \subset \hat \O$ with the Lipschitz boundary
$\partial \O = \Gamma_{f} \cup \Gamma_{\O}$ where $\Gamma_{\O}$ is the free component of the boundary that is to be determined (see \cref{fig:problem} for an example). For simplicity, we assume there is no intersection between $\Gamma_\O$ and $\Gamma_f$.
\begin{figure}[h]\label{fig:problem}
\centering
\includegraphics[width=0.50\textwidth]{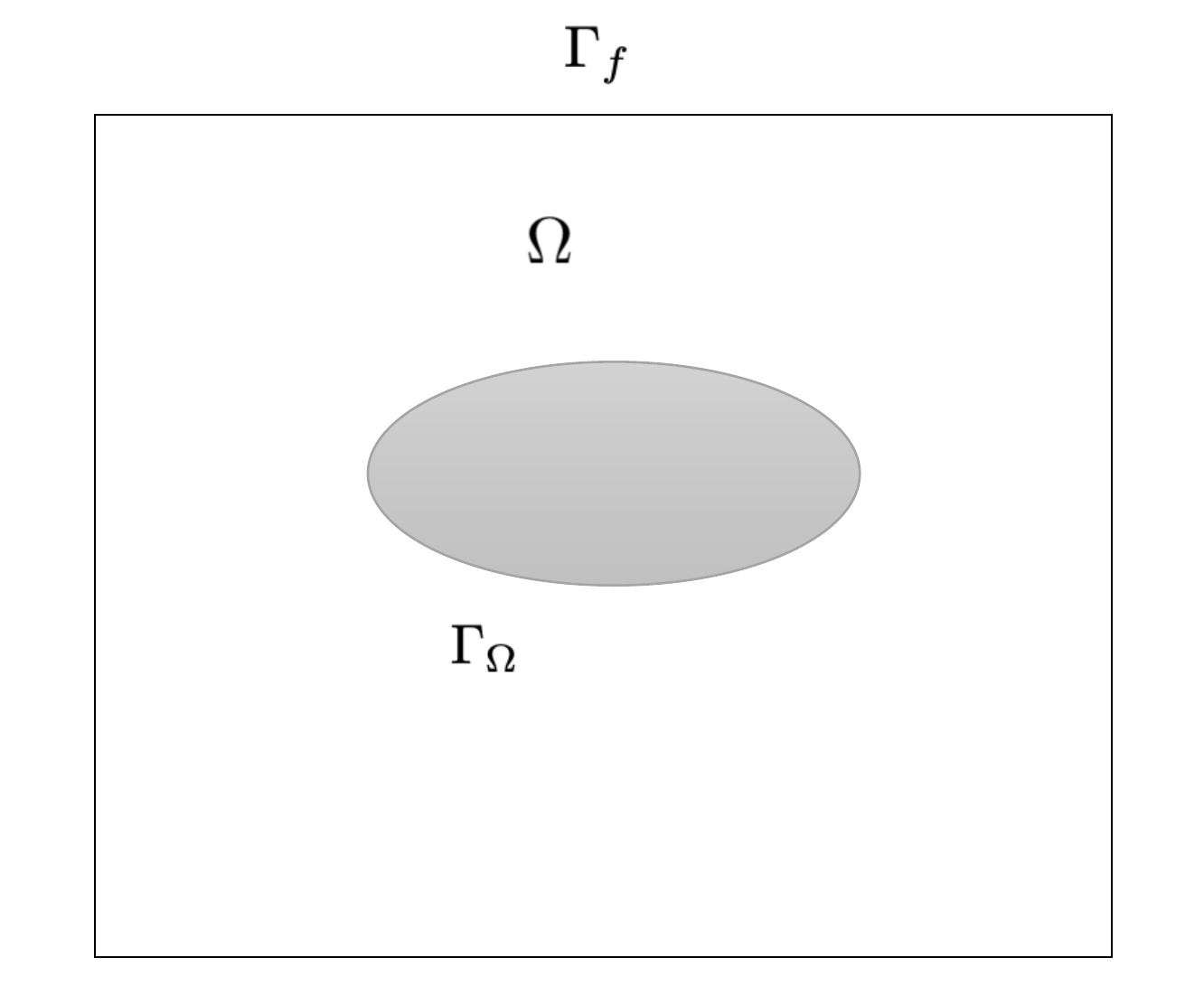}
\caption{The domain with the fixed boundary $\Gamma_{f}$ and the unknown boundary $\Gamma_{\O}$}.
\end{figure}
We consider the interior type ill-posed free boundary Bernoulli problem, i.e., the fixed boundary $\Gamma_{f}$ is exterior to $\Gamma_{\O}$.
Find $\tilde \Omega \in \mathcal{O}$ and $u: \tilde \Omega \rightarrow R$ such that 
\begin{equation}\label{inverse-problem}
\begin{split}
	- \triangle  u = f & \mbox{ in } \tilde \O,\\
	u= 0 &\mbox{ on } \Gamma_{\tilde \O},\\
	u = g_D & \mbox{ on } \Gamma_{f},\\
	D_n u = g_N & \mbox{ on } \Gamma_{f}.
\end{split}
\end{equation}
Here $ \mathcal{O}$ denotes the set of all admissible domains.
The datum $(f,g_D,g_N)$ is chosen such that $f \in L^2(\tilde \O)$, $g_D \in H^{1/2}(\Gamma_{f})$ and $g_N \in H^{-1/2}(\Gamma_{f})$. Here $D_n u  = \nabla u \cdot \bn$ where $\bn$ is the unit outer normal vector to the domain. 
It is known that, provided the data $f,\, g_D ,\, g_N$ are compatible with a solution $\Gamma_{\O}$, this solution is unique. This follows by a unique continuation argument from the Cauchy data on $\Gamma_{f}$. For a proof in the context of scattering problems we refer to \cite[Theorem 2]{CK18}.

For an arbitrary $\Omega \in \mathcal{O}$, the system \cref{inverse-problem} is over-determined and therefore the solution may not exist. 
To represent the interface, we here use the zero level set of a continuous function. The value of level set function away from the interface is not important, provided the gradient of the level set function do not degenerate. To be precise, for each $\O$ we aim to find a level set function $\phi(\Omega)$ such that 
\begin{equation}
	\phi(x) 
	\begin{cases}
	> 0 &\mbox{ if } x \not \in \O,\\
	= 0 & \mbox{ if } x \in \Gamma_{\O},\\
	<0 & \mbox{ if } x \in \O.
	\end{cases}
\end{equation}
To locate the true free boundary starting from an initial guess $\O$, we use a shape optimization procedure that uses a well-posed pair of forward and dual problems.  The free boundary is then transported in the optimal direction using an interface transport direction given by the shape derivative of the cost functional. 



Define the spaces
\begin{align}
	H^1_{0,\Gamma_{\O}}(\O)& := \{v \in H^1(\O): v =0 \mbox{ on } \Gamma_{\O}\}
\\	
	H^1_{0}(\O) &:= \{v \in H^1(\O): v =0 \mbox{ on } \partial \O \}.
\end{align}
Let $(\cdot,\cdot)_\Omega$ denote the $L^2$-scalar product over $\Omega \subset \mathbb{R}^2$ and $\left<\cdot,  \cdot\right>_{\Gamma}$ the $L^2$-scalar product over 
the curve $\Gamma \subset \mathbb{R}^2$. The $L^2$-norm over a subset $X$ of $\mathbb{R}^s$, $s=1,2,$ will be denoted by
$\|\cdot\|_X$.
 
To find an approximation of the solution to the inverse problem \cref{inverse-problem}, we solve the following PDE constrained optimization problem:
find $\O^* \in \mathcal{O}$ such that
\begin{equation}\label{optimization-pro}
	J(\O^*) = \min \limits_{\O \in \mathcal{O}} J(\O) \quad \forall \, \O \in \mathcal{O},
\end{equation}
where the cost functional is defined by
\begin{equation}\label{cost-functional}
J(\O) = \dfrac{1}{2}h^{-1}\| g_D -u(\O)\|^2_{\Gamma_{f}},
\end{equation}
where $h$ is the mesh size of the finite element mesh that will be used for the numerical approximation, and $u(\O)\in H_{0,\Gamma_{\O}}(\O)$ satisfies 
\begin{equation}\label{u-formulation}
a(u, v):= (\nabla u, \nabla v) = f(v) + \left<g_N, v\right>_{\Gamma_{f}} \quad \forall \, v \in H_{0,\Gamma_{\O}}(\O).
\end{equation}
When there is no risk of ambiguity, we replace $u(\O)$ by $u$. 

The corresponding Lagrangian for the constrained minimization problem \cref{optimization-pro} can be formalized as
follows:
\begin{equation}\label{Lagrangian}
\begin{split}
\mathcal{L}(\O, u(\O), v) = \dfrac{1}{2}h^{-1}\| g_D -u(\O)\|^2_{\Gamma_{f}} - a(u,v) + l(v)
\end{split}
\end{equation}
where $l(v) = f(v)+ \left< g_N, v\right>_{\Gamma_{f}}$.

To find the critical point, denoted by $(u,p)$, we take the Fr\'echet derivative with respect to $u$ and $v$.
For the primal variable, $u$, it yields to solve \cref{u-formulation}.
For an arbitrary  $\O \in \mathcal{O}$,
this corresponds to the following forward problem: find $u(\O): \O \rightarrow R$ such that
\begin{equation}\label{u-forward}
\begin{split}
	- \triangle  u = f & \mbox{ in } \O,\\
	u= 0 &\mbox{ on } \Gamma_{\O},\\
	D_n u = g_N & \mbox{ on } \Gamma_{f}.
\end{split}
\end{equation}

For the adjoint solution $p$, we obtain the weak formulation: find $p \in  H_{0,\Gamma_{\O}}^1(\O)$ such that
\begin{equation}\label{p-formulation}
\begin{split}
	 (\nabla v, \nabla p)_{\O} = h^{-1}\left< u - g_D, v\right>_{\Gamma_{f}} \; \forall v \in H_{0,\Gamma_{\O}}^1(\O).
\end{split}
\end{equation}



\begin{remark}
If $\O = \tilde \O$ we have $u = g_D$ on $\Gamma_{f}$ and hence $p \equiv 0$ in $\tilde \O$.
\end{remark}

\begin{remark}
The relation between \cref{inverse-problem} and \cref{optimization-pro} is as follows. If $\tilde \O$ is the solution to \cref{inverse-problem} and $\tilde \O \in \mathcal{O}$ then $\tilde\O$ is the global minimum to \cref{optimization-pro}. The converse is also true, by the uniqueness of the inclusion, however there may be local minima that complicate the identification.
\end{remark}

\section{Approximation of primal and dual solutions using CutFEM}\label{sec:cutfem}
In this section we approximate the primal and dual solution for \cref{u-formulation} and  \cref{p-formulation}, respectively.
To solve the primal and dual solutions we use the CutFEM method.  The main advantages of using the CutFEM method is that a fixed background mesh of $\hat \O$ may be used that does not need to fit the moving boundary. The background domain $\hat \O$ is chosen to be a regular domain, e.g., unit square, such that $\O \in \hat \O$ for all $\O \in \mathcal{O}$.
Moreover, stability and accuracy of CutFEM, similar to standard FEM is guaranteed given proper stabilization. 

Let $\cT =\{K\}$ be a shape regular triangular partition of $\hat \O$ and $h = \max \limits_{K \in \cT} h_K$ where $h_K$ is the diameter of $K$.
Define
\[
	V_{h}(\O) = \{ v \in H_1(\O): v|_K \in P_1(K) \;\forall \, K \in \cT \},
\]
and, for $v,w \in V_{h}(\O)$, define
\begin{equation}\label{a-h}
a_h(w,v) := \tilde{a}_h(w,v) + j(w,v)  
\end{equation}
with
\begin{equation}
\tilde{a}_h(w,v) = (\nabla w, \nabla v)_{\O} - \left<D_n w, v\right>_{\Gamma_{\O}}  - \left<D_n v, w\right>_{\Gamma_{\O}} + 
\beta h^{-1} \left<w,v \right>_{\Gamma_{\O}} ,
\end{equation}
and
\begin{equation}
j(w,v) = \sum_{F \in \cE_I} \gamma h \int_{F} \jump{D_n w} \jump{D_n v}\,ds,
\end{equation}
where $\cE_I = \{ F \subset \partial K: K \in \cT ;\; F \cap \partial \hat \Omega \ne F \}$ denotes the set of interior faces of the background mesh. The form $j(w,v)$ is the so-called ghost penalty stabilization \cite{burman2010ghost} and $\jump{\cdot}|_F$ denotes the jump operator on $F$. To simplify the presentation, we here make the ghost penalty stabilization act on all the interior faces. In practice it may be localized to the element faces in the interface zone.

Considering the following variational problems:
find $u_{h}\in V_{h}(\O)$
 such that
\begin{equation}\label{u-h}
a_h(u_{h},v) = (f,v)_\O+\left<g_N, v\right>_{\Gamma_{f}} 
\quad \forall v \, \in V_h(\O),
\end{equation}
find $p_{h} \in V_h(\O)$ such that
\begin{equation}\label{p-h}
a_h(p_{h},v) =  h^{-1}\left<u_h  -g_D , v\right>_{\Gamma_{f}} \quad \forall \,v \, \in V_h(\O).
\end{equation}

\begin{remark}
Note that in the above formulations all Dirichlet boundary conditions are imposed weakly using Nitsche's method \cite{nitsche1971variationsprinzip}.
\end{remark}


\section{Shape derivatives}\label{sec:shape-derivatives}
In this section, we aim to derive the formulas for different types of shape derivatives. We will first discuss some basic definitions and derive shape derivatives for bulk quantities, this is standard textbook material and essentially follows \cite{SZ92,delfour2011shapes}. Then we extend these arguments to functionals defined on lower dimensional subsets, that are useful for the approximation of the shape derivative of the CutFEM formulation.

\subsection{Definition of the shape derivative}
For $\O \in \mathcal{O}$, we let $W(\O, \mathbb{R}^2)$ denote the space of sufficiently smooth vector fields $\bftheta: \Omega \rightarrow \mathbb{R}^2$ such that $\bftheta \equiv 0$ on $\Gamma_{f}$.  For a vector field , $\bftheta \in W(\O, \mathbb{R}^2)$, we define the map
\begin{equation}
T_{t, \bftheta}: x \in \O  \rightarrow x + t \bftheta(x) \in \O_t(\bftheta) \subset \mathbb{R}^2.
\end{equation}
The variable $t$ is interpreted as the pseudo-time. For small $t$ the mapping $\O \rightarrow \O_t(\bftheta)$ is assumed to be a bijection. We also assume that $\O_t(\bftheta) \in \mathcal{O}$ for any $t \in I=\{-\delta, \delta\}$, with $\delta>0$ small enough. When there is no risk of confusion, we let $\O_t = \O_t(\bftheta)$. 

The shape derivative of the cost functional $J(\O)$ with respect to the domain $\O$ in the direction of $\bftheta$ is defined as
\begin{equation}
D_{\O, \bftheta} J(\O) := \lim_{t \to 0} \dfrac{1}{t}(J(\O_t(\bftheta)) - J(\O)). 
\end{equation}

For a scalar function $v(x,t): \O \times I \rightarrow \mathbb{R}$ that is smooth enough, we define the material  derivative in the direction $\bftheta$ by
\begin{equation}
D_{t,\bftheta} v(x) = \lim_{t \rightarrow 0} \dfrac{v(x(t),t) - v(x(0),0)}{t}
\end{equation} 
where $x(t) = T_{t, \bftheta}(x)) = x + t \bftheta(x)$ and $x(0) = x$.
We also define the pseudo-time derivative by 
\begin{equation}
	\partial_t v(x)= \lim_{t \rightarrow 0} \dfrac{v(x,t) - v(x,0)}{t}.
\end{equation} 
By the chain rule it is easy to see that
\begin{equation}\label{material-derivative-to-}
D_{t,\bftheta}\, v = \partial_t v + \bftheta \cdot \nabla v.
\end{equation}
The product rule holds for the material derivative:
\begin{equation}
D_{t,\bftheta}\, (vw) = w D_{t,\bftheta} v + v D_{t,\bftheta}\, w.
\end{equation}
For future reference, we introduce the notation $\dot{v} := D_{t,\bftheta} v$ and $v':= \partial_t v$.


\begin{lemma} \label{lem:shape-derivative}
Let $\O$ be an open set in $\mathbb{R}^2$ and $\bftheta: \mathbb{R}^{2} \rightarrow  \mathbb{R}^{2}$ be an injective differentiable mapping. Then the following equalities hold:
\begin{equation}\label{shape-derivative-on-domain-and-boundary}
\begin{split}
	D_{\O,\bftheta} \int_{\O} \phi \, dx &= \int_\O (\dot{ \phi} + (\nabla \cdot \bftheta) \phi) \,dx\\
	D_{\O,\bftheta} \int_{\Gamma} \psi \, ds &= \int_{\Gamma} (\dot{ \psi} + 
	(\nabla_{\Gamma} \cdot \bftheta) \psi) \,ds
\end{split}
\end{equation} 
where we assume that $\phi(x,t),\psi(x,t): \mathbb{R}^2 \times I \rightarrow \mathbb{R}$ are functions smooth enough for the expressions of \eqref{shape-derivative-on-domain-and-boundary} to be well defined and
where $\nabla_{\Gamma}\cdot \bftheta = \nabla \cdot \bftheta - \bfn \cdot D\bftheta \cdot \bfn^t $.
\end{lemma}
\begin{proof}
We give a brief sketch of the proof below to make the presentation self contained. This exposition follows the arguments in \cite{delfour2011shapes}.
\[
	\int_{\O_t(\bftheta)} \phi(x,t) \,dx = \int_{\O} \phi \circ T_{t, \bftheta} \mu_t\,dx=
	\int_{\O} \phi((x(t), t)  \mu(t)\,dx
\]
where $\mu(t) = \det (DT_{t, \bftheta})$ and $x(t) = x + t \bftheta(x)$. Note that $\mu(0)=1$.
By definition we have
\begin{equation}
\begin{split}
	 D_{\O,\bftheta} \int_{\O} \phi \, dx=& \lim_{t \to 0} \dfrac{1 } {t} \left(\int_{\O_{t}(\bftheta)} \phi(x,t) \,dx - \int_{\O} \phi(x,0) \,dx \right)\\
	=&\lim_{t \to 0} \int_{\O}\dfrac{1}{t} \left( \phi(x(t), t)\mu_t - \phi(x, 0)\mu_0 \right)\,dx\\
	=&\int_\Omega  \dot{\phi}(x,0) dx + \int_\Omega  \phi(x,0) \divvr \bftheta dx
\end{split}
\end{equation}
where we have used the fact that (see Example 3.1 in \cite{delfour2011shapes}) 
\[
	\lim_{t \to 0} \dfrac{1}{t}(\mu(t) - \mu(0))= \divvr \bftheta.
\]
To prove the second part of \cref{shape-derivative-on-domain-and-boundary} we have
that
\[
	\int_{\Gamma_{\O_t( \bftheta)}} \phi(x,t) \,dx = \int_{\Gamma_\O} \phi \circ T_{t, \bftheta} \o(t)\,dx=
	\int_{\Gamma_\O} \phi(x(t), t)  \o(t)\,dx
\]
where
$\o(t) = \mu(t) |(DT_{t,\bftheta})^{-t} \cdot \bfn|$. Note that $\o(0)=1$.  Finally, combining the fact that
\[
	\lim_{t \to 0}\dfrac{1}{t}(\o(t) - \o(0)) = \divvr \bftheta - (D\bftheta \cdot \bfn)  \cdot \bfn
\]
gives the second part of \cref{shape-derivative-on-domain-and-boundary}. This completes the proof of the lemma.
\end{proof}

\begin{lemma}\label{lem:shape-derivative1}
The following relation holds:
\begin{equation}\label{eq:domegathera}
\begin{split}
D_{\O, \bftheta} \int_{\O} \nabla w \cdot \nabla v \,dx 
=&	    \int_{\O} (\nabla \cdot \bftheta) \nabla w \cdot \nabla v   
	    - \nabla w \cdot (D \bftheta + (D\bftheta)^t) \nabla v \,dx\\
&	     + \int_{\O} \nabla \dot{w} \cdot \nabla v  +  \nabla \dot{v} \cdot \nabla w \,dx,
\end{split}
\end{equation}
where we assume that $w(x,t), v(x,t): \mathbb{R}\times I \rightarrow \mathbb{R}$ are functions smooth enough for \eqref{eq:domegathera} to be well defined.
\end{lemma}
\begin{proof}
By change of variables, we have
\begin{equation}
	\begin{split}
	\lim_{t \to 0} 
		 & \dfrac{1}{t} \left( \int_{\O_t(\bftheta)} \nabla w(x,t) \cdot \nabla v(x,t) \,dx - \int_{\O} \nabla w(x,0) \cdot \nabla v(x,0) \,dx  \right)\\
		 =&\lim_{t \to 0} \dfrac{1}{t} \left( \int_{\O} ((\nabla w \circ T_t) \cdot (\nabla v \circ T_t)  \mu(t) \,dx 
		 - \int_{\O} \nabla w(x,0) \cdot \nabla v(x,0) \,dx  \right)\\
		=&\lim_{ t \rightarrow 0} \dfrac{1}{t} 
		\left( \int_{\O} \left( A(t) \cdot\nabla (w \circ T_t)\right) \cdot \nabla (v \circ T_t) \,dx  
		 -   \int_{\O} \nabla w \cdot \nabla v \,dx  \right)\\
	  =&
	    \int_{\O} ( A'(t)\cdot  \nabla w ) \cdot  \nabla v 
	     + \nabla \dot{w} \cdot \nabla v  +  \nabla \dot{v} \cdot \nabla w \,dx,
	\end{split}
\end{equation}
where we used the chain rule
\[
 (\nabla u) \circ T_t  = DT_{t, \bftheta}^{-t} \cdot \nabla (u \circ T_t)
 \]
and introduced $A(t)$ and its derivative
 \begin{equation}\label{dAdx}
	A(t) =\mu(t) DT_t^{-1} (DT_t)^{-t}, \quad
	\quad
	A'(t) = \nabla \cdot \bftheta I - (D\bftheta + (D\bftheta)^t),
\end{equation}
and finally we employed the product rule. This completes the proof of the lemma.
\end{proof}

\subsubsection{Shape derivatives of boundary and face terms}
For the sake of simplicity, we denote by $S(\bftheta) =D \bftheta + (D \bftheta)^t. $
\begin{lemma}\label{lem:shape-derivative4}
The following relation holds:
\begin{equation}\label{eq:DGammazero}
\begin{split}
D_{\O, \bftheta} \int_{\Gamma_{\O}} (D_n w) v \,ds
&=\int_{\Gamma_{\O}} (( \nabla \cdot \bftheta ) (D_n w)  v
	   - (S(\bftheta) \cdot \nabla w) \cdot \bfn v\,ds
\\
&\qquad +	\int_{\Gamma_{\O}} 
	   	    (D_n \dot w ) v \,ds
	   + (\nabla  w \cdot \bfn) \dot v \,ds.
\end{split}
\end{equation}
where we assume that $w(x,t), v(x,t): \mathbb{R}\times I \rightarrow \mathbb{R}$ are functions smooth enough for \eqref{eq:DGammazero} to be well defined. 
\end{lemma}
\begin{proof}
First by change of variable we have 
	\begin{equation}
	\begin{split}
	 \int_{\Gamma_{\O_t}} \nabla w(x,t)\cdot \bfn_t v(x,t) \,ds &= 
	 \int_{\Gamma_{\O}} (\nabla w \circ T_t) \cdot (\bfn_t \circ T_t) (v \circ T_t) \o(t)\,ds \\
	 &=
	  \int_{\Gamma_{\O}} (DT_t^{-t} \cdot \nabla (w \circ T_t)) \cdot (\bfn_t \circ T_t) (v \circ T_t) \o(t)\,ds.
	 \\
	\end{split}
\end{equation}
From Theorem 4.4 in \cite{delfour2011shapes} it holds that
\[
	\bfn_t \circ T_t = \dfrac{DT_t^{-t} \cdot \bfn}{|DT_t^{-t} \cdot \bfn|}.
\]
Recall that $\o_t = \mu(t) |DT_t^{-t} \cdot \bfn|$. By a direct calculation we have
	\begin{equation}\label{normal-sd-a}
	\begin{split}
	 & \int_{\Gamma_{\O_t}} (\nabla w(x,t)\cdot \bfn_t) v(x,t) \,ds 
	 =
	  \int_{\Gamma_{\O}} (A(t) \cdot \nabla (w \circ T_t)) \cdot \bfn (v \circ T_t)\,ds
	\end{split}
	\end{equation}
	 Finally, combing \cref{normal-sd-a} and \cref{dAdx} gives
	 \begin{equation}
	 	\begin{split}
		&D_{\O, \bftheta} \int_{\Gamma_{\O}} \nabla w \cdot \bfn v \,ds
	 =
	  \int_{\Gamma_{\O}} (A'(t) \cdot (\nabla w \cdot \bfn)  v\, + (\nabla \dot w \cdot \bfn) v \,ds
	   + (\nabla w \cdot \bfn) \dot v \,ds\\
	   = &
	   \int_{\Gamma_{\O}} (( \nabla \cdot \bftheta ) (\nabla w \cdot \bfn)  v
	   - (S(\bftheta) \cdot \nabla w) \cdot \bfn v
	   	    + (\nabla w \cdot \bfn) \dot v \,ds
	   + (\nabla \dot w \cdot \bfn) v \,ds.
	\end{split}
\end{equation}
This completes the proof of the lemma.
\end{proof}
The stability of the CutFEM method is ensured by the ghost penalty term. In the following Lemma we give a result allowing the integration of the effect of this term in the shape gradient. The proof is given in the appendix.

\begin{lemma}\label{lemma4.4}
Assume that $w,v \in H^1(\Omega,t)$ and that locally on each triangle $K$, $w(x,t)\vert_K,\,v(x,t)\vert_K \in H^{3/2+\epsilon}(K)$. Then there holds
\begin{equation}
	D_{\O,\bftheta}\int_F \jump{D_n w} \jump{D_n v} \,ds 
	=  \int_{F} ( \jump{\nabla \dot w \cdot \bn} \jump{\nabla  v\cdot \bn} 
	+  \jump{\nabla w \cdot \bn} \jump{\nabla  \dot v\cdot \bn} ) \,ds + \epsilon_F(w,v)
\end{equation}
where
\begin{equation}
\begin{split}
	\epsilon_F(w,v)
	=
	& \int_{F} \jump{ (\nabla \cdot \bftheta) \nabla w \cdot \bn 
	- \nabla w \cdot S(\bftheta)  \cdot \bn} \jump{\nabla  v\cdot \bn} \,ds \\
	&+ \int_{F} \jump{ (\nabla \cdot \bftheta) \nabla v \cdot \bn
	 - \nabla v \cdot S(\bftheta) \cdot \bn} \jump{\nabla  w\cdot \bn} \,ds \\
	& - \int_{F} \jump{\nabla w \cdot \bn} \jump{\nabla  v\cdot \bn} \, (\divvr \bftheta - (D\bftheta \cdot \bfn)  \cdot \bfn)ds.
\end{split}
\end{equation}
\end{lemma}

\subsection{Optimize-then-discretize approach}
In this subsection we first analyse the shape optimization based on the optimize-then-discretize approach, i.e., the representation for the shape derivative is computed based 
on the continuous problems. In the numerical approximation, we will simply replace the continuous solutions by the numerical ones. Note that for this approach the formula of the shape derivative is then independent of the numerical method used to approximate the solutions. Therefore, the shape derivative is not exact since by assumption its input is assumed to be the true solutions while in reality it is evaluated using their approximations. The error in the gradient will be of optimal order asymptotically, if the CutFEM solution has optimal error estimates in $W^{1,4}(\Omega)$ and $L^4(\Omega)$, see \cite{BEHLL17}. 

On $\O_t(\bftheta)$, $t \in [0, \tau]$ we define $u(x,t) \in H_{0, \Gamma_{\O_t}}^1$ and $p(x,t) \in H_{0, \Gamma_{\O_t}}^1$ such that
\begin{equation}
	(\nabla u(x,t), \nabla v)_{\O_t} = (f, v)_{\O_t} +\left< g_N, v\right>_{\Gamma_{f}}
	\quad \forall \, v \in H_{0, \Gamma_{\O_t}}^1
\end{equation}
and
\begin{equation}
	(\nabla v, \nabla p(x,t))_{\O_t} = h^{-1}\left<u(x,t) - g_D, v\right>_{\Gamma_{f}} 
	\quad \forall \, v \in H_{0, \Gamma_{\O_t}}^1.
	\end{equation}
Immediately we have that  
$\dot p = \dot u = 0$ on $\Gamma_{\O}$, therefore $\dot u \in H_{0,\Gamma_{\O}}(\O)$ and $\dot p \in H_{0,\Gamma_{\O}}(\O)$.

\begin{lemma}\label{lemma:sd}
Let $ \mathcal{L}$ be defined in \cref{Lagrangian}. Then its shape derivative has the following representation:
\begin{equation}\label{shape-derivative-formula}
\begin{split}
&D_{\O, \bftheta} \mathcal{L}(\O, u(\O), p(\O)) \\
&\;= \int_{\O}(\divvr \bftheta) \left(f p - \nabla u \cdot \nabla p  \right) \,dx
+  \int_{\O} \nabla u \cdot  S(\bftheta) \cdot \nabla p \,dx
+ \int_{\O} (\nabla f \cdot \bftheta) p \,dx.
\end{split}
\end{equation}
\end{lemma}

\begin{proof}
Rearrange $ \mathcal{L}(\O, u, p)$ such that 
\begin{equation}\label{domain-derivative-00}
	 \mathcal{L}(\O, u, p) \triangleq \mathcal{A}_1 + \mathcal{A}_2
\end{equation}
where
\begin{align*}
\mathcal{A}_1 =   -(\nabla u, \nabla p)_\O+ \left( f, p\right)_{\O}, \quad 
\mathcal{A}_2 =
	\dfrac{1}{2} h^{-1} \left<g_D- u, g_D- u\right>_{\Gamma_{f}} + \left<g_N, p\right>_{\Gamma_{f}}.
\end{align*}
By \cref{lem:shape-derivative} and \cref{lem:shape-derivative1}, we firstly have
\begin{equation}\label{domain-derivative-0}
\begin{split}
D_{\bftheta,\O} \mathcal{A}_1
=&-(\divvr \bftheta, \nabla u \cdot \nabla p - fp)_\O
+   \int_\O \nabla u \cdot S(\bftheta) \cdot \nabla p  \,dx\\
&- ( \nabla \dot{u} , \nabla p)_\O
- (\nabla u,\nabla \dot{p})_\O  +(\dot f, p)_\O + (f, \dot p)_\O.
\end{split}
\end{equation}
Note that $\dot f = \nabla f \cdot \bftheta$ since $f'=0$.
Thanks to the fact that $\dot u \in H_{0,\Gamma_{\O}}(\O)$ and $\dot p \in H_{0,\Gamma_{\O}}(\O)$, by  \cref{u-formulation} and
\cref{p-formulation} we have
\begin{equation}\label{domain-derivative-1}
\begin{split}
- ( \nabla \dot{u} , \nabla p)_\O
- (\nabla u,\nabla \dot{p})_\O   + (f, \dot p)_\O
&= -h^{-1}\left< u -g_D, \dot u \right>_{\Gamma_{f}} - \left< g_N, \dot p \right>_{\Gamma_{f}}
\\
&
= -h^{-1}\left< u -g_D, u' \right>_{\Gamma_{f}} - \left< g_N, p' \right>_{\Gamma_{f}}.
\end{split}
\end{equation}
Note that on $\Gamma_{f}$, we have used the fact that $\dot u = u'$ and $\dot p = p'$, since 
$\bftheta =0$ on $\Gamma_{f}$.
By the product and chain rule we immediately have
\begin{equation}\label{bd-derivative-0}
\begin{split}
&D_{\bftheta,\O}\mathcal{A}_2 = h^{-1} \left<u - g_D, u'\right>_{\Gamma_f} +  \left< g_N, p' \right>_{\Gamma_{f}}.
\end{split}
\end{equation}
Combining the identities gives \cref{shape-derivative-formula}. This completes the proof of the lemma.
\end{proof}

\subsection{Discretize-then-optimize approach}
In this subsection we aim to derive the shape derivative formula for the approach where we first discretize the Lagrangian using the CutFEM method and then we evaluate the shape derivative of the discrete functional. For this case the optimization analysis is dependent on the numerical solutions and the numerical method that is used. As a consequence the shape derivative is exact for the discrete functional.

Starting from the Lagrangian \cref{Lagrangian} we obtain the discrete Lagrangian form 
by replacing the bilinear and linear forms $a$ and $l$ by the corresponding discrete forms $a_h$ (defined in \cref{a-h}) and $l_h(v)$: 
\begin{equation}\label{Lagrangian-1}
\begin{split}
\mathcal{L}_h (\O, w_h, v_h) = \dfrac{1}{2}h^{-1}\| g_D -w_h\|^2_{\Gamma_{f}} - a_h(w_h,v_h) + l_h(v_h).
\end{split}
\end{equation}

Note that taking the Fr\'echet derivative with respect to $v_h$ and $w_h$ in \cref{Lagrangian-1} gives exactly the CutFEM formulation for $u_h$ and $p_h$ in \cref{u-h} and \cref{p-h}, respectively.

To define the shape derivative, firstly we need to define the function space for $u_h(x,t)$ and $p_h(x,t)$ on $\O_t$. We do this by using a pullback map to $\O$ where the elements are triangular and use the standard definition of the finite element space on the reference domain.

For each $K \in \cT_h$, let $K^t = T_{t,\bftheta} K$. When there is no risk of ambiguity, we replace 
$T_{t,\bftheta}$ by $T_t$.
Here we assume that 
 $T_t \in [C^1(\O)]^d$. Then, by the inverse function theorem, $T_t$ is a bijection for sufficiently small $t$ and its derivatives are point wise well defined. We also define 
 $\cT_h^t := \{K^t, K \in \cT_h\}$ and 
 \[
 	V_h^t(\O_t):=\{ v \in H^1(\cT_h^t), v|_{K^t} \in V_h^t(K^t)\}
 \]
 where $V_h^t(K^t)$ satisfies $V_h^t(K^t) =V_h(K) \circ T_t^{-1}$.
 It is then easy to verify that 
 \[ v_h^t \circ T_t \in V_h(\O) \quad \forall v_h^t \in V_h^t(\O_t).\]
 We now define $u_h(x,t)$ and $p_h(x,t)$ on $\O_t$.
 Let
 $u_h(x,t)$ and $p_h(x,t)$  be the solution of  \cref{u-h} and \cref{p-h}
in the mapped space $V_h^t(\O_t)$ using integrals over $\O_t$ and 
$\Gamma_{\O_t}$, respectively instead of $\O$ and $\Gamma_\O$.

\begin{lemma}
	Let $u_h(x,t)$ and $p_h(x,t)$ be defined as above. Then 
	\begin{equation}
		\dot u_h \in V_h(\O) \quad \mbox{and} \quad \dot p_h \in V_h(\O).
	\end{equation}
\end{lemma}
\begin{proof}
By the definition, we have that
\begin{equation}
\begin{split}
	\dot u_h(x) = &\lim_{ t \to 0}\dfrac{1}{t} (u_h(x(t), t) - u_h(x,0) )\\
	=&  \lim_{ t \to 0}\dfrac{1}{t} (u_h(T_t (x), t) - u_h(x,0) ).
\end{split}
\end{equation}
Since both $u_h(T_t (x), t) \in V_h$ and $u_h(x,0) \in V_h$, we have that
$\dot u_h \in V_h$. The result for $\dot p_h$ holds by the same argument.
\end{proof}


In the following lemma we derive the integral representation for the shape derivative of 
$\mathcal{L}_h$.

\begin{lemma}\label{lemma:sd-h1}
Let $ \mathcal{L}_h$ be defined in \cref{Lagrangian-1}. Then its shape derivative has the following representation:
\begin{equation}\label{shape-derivative-formula-h1}
\begin{split}
&D_{\O, \bftheta} \mathcal{L}_h(\O, u_h(\O), p_h(\O)) \\
&\quad= \int_{\O}(\divvr \bftheta) \left(f p_h - \nabla u_h \cdot \nabla p_h  \right) \,dx
\\
&\quad \quad 
+   \int_\O (\nabla u_h) S(\bftheta)(\nabla p_h)^t \,dx
+ \int_{\O} (\nabla f \cdot \bftheta) p_h \,dx
\\
&\quad \quad +\int_{\Gamma_{\O}} ( \nabla \cdot \bftheta ) (D_n u_h)  p_h
	   - (S(\bftheta)  \cdot \nabla u_h) \cdot \bfn p_h
	   	    \,ds\\
&\quad \quad  +\int_{\Gamma_{\O}} ( \nabla \cdot \bftheta ) (D_n p_h)  u_h
	   - (S(\bftheta) \cdot \nabla p_h) \cdot \bfn u_h
	   	    \,ds\\
&\quad \quad -
 \int_{\Gamma_{\O}} \beta h^{-1} ( \nabla_{\Gamma}\cdot \bftheta) u_hp_h 
 + \textcolor{black}{ \sum_{F \in \cE_I} \gamma h \epsilon_F(u_h, p_h)}\,
\end{split}
\end{equation}
\end{lemma}

\begin{proof}
Rearrange $ \mathcal{L}(\O, u, p)$ such that 
\begin{equation}\label{domain-derivative-00-a}
	 \mathcal{L}(\O, u, p) \triangleq \sum_{i=1}^4\mathcal{A}_i
\end{equation}
where
\begin{align*}
\mathcal{A}_1 &=   -(\nabla u_h, \nabla p_h)_\O+ \left( f, p_h\right)_{\O},
\\
\mathcal{A}_2 &=
	\dfrac{1}{2} h^{-1}  \left<g_D- u_h, g_D- u_h \right>_{\Gamma_{f}} + \left<g_N, p_h\right>_{\Gamma_{f}},
	\\
	\mathcal{A}_3 &= \left< D_n u_h, p_h\right>_{\Gamma_{\O}}+
	\left< D_n p_h, u_h\right>_{\Gamma_{\O}}
	-
	\beta h^{-1}\left< u_h, p_h\right>_{\Gamma_{\O}},
	\\
	\mathcal{A}_4 &=-j(u_h, p_h).
\end{align*}
For the first two terms, we could derive its shape derivative similarly as in \cref{lemma:sd}:
\begin{equation}\label{domain-derivative-00-b}
\begin{split}
D_{\bftheta,\O} \mathcal{A}_1
=&- \int_\O(\divvr \bftheta)(\nabla u_h \cdot \nabla p_h - fp_h ) \,ds
+   \int_\O (\nabla u_h) S(\bftheta)(\nabla p_h)^t \,dx
+ \int_\O(\nabla f \cdot \bftheta) p_h \,dx  \\
&
- \int_\O ( \nabla \dot u_h \cdot \nabla p_h) \,dx
- \int_\O(\nabla u_h \cdot \nabla \dot p_h) \,dx  +  \int_\O(f \dot p_h) \,dx.
\end{split}
\end{equation}
Similarly, we have
\begin{equation}\label{bd-derivative-0-a}
\begin{split}
&D_{\bftheta,\O}\mathcal{A}_2 =- h^{-1} \left<g_D - u_h, \dot u_h\right>_{\Gamma_{f}} +  \left< g_N, \dot p_h \right>_{\Gamma_{f}}.
\end{split}
\end{equation}
For $\mathcal{A}_3$, by \cref{lem:shape-derivative4}  we have
\begin{equation}\label{bd-derivative-0-b}
\begin{split}
D_{\bftheta,\O}\mathcal{A}_3  =&\int_{\Gamma_{\O}} ( \nabla \cdot \bftheta ) (D_n u_h)  p_h
	   - (S(\bftheta) \cdot \nabla u_h) \cdot \bfn p_h
	   	   + (D_n \dot u_h ) p_h 
	   + (D_n u_h) \dot p_h \,ds\\
&+\int_{\Gamma_{\O}} (( \nabla \cdot \bftheta ) (D_n p_h)  u_h
	   - (S(\bftheta) \cdot \nabla p_h) \cdot \bfn u_h
	   	   + (D_n \dot p_h ) u_h 
	   + (D_n p_h ) \dot u_h \,ds\\
&-
 \beta h^{-1}  \int_{\Gamma_{\O}} ( \nabla_{\Gamma}\cdot \bftheta) u_hp_h
	+  \dot u_h p_h 	
	+ u_h \dot p_h \,ds.
\end{split}
\end{equation}

And for $\mathcal{A}_4$ by \cref{lemma4.4} we have
\begin{equation}
\begin{split}
&D_{\bftheta,\O}\mathcal{A}_4 = - j(u_h, \dot p_h) - j(\dot u_h, p_h) - \sum_{F\in\cE_I} \gamma \epsilon_F(u_h, p_h).
\end{split}
\end{equation}

Thanks to the fact that $\dot u_h \in V_h$ and $\dot p_h \in V_h$, by  \cref{u-h} and
\cref{p-h} with $v$ replaced by $\dot p_h$ in \cref{u-h} and by $\dot u_h$ in \cref{p-h} we have
\begin{equation}\label{domain-derivative-0-c}
\begin{split}
0&=- ( \nabla \dot u_h, \nabla p_h)_\O
- (\nabla u_h,\nabla \dot p_h)_\O  + (f, \dot p_h)_\O\\
&- h^{-1} \left<g_D - u_h, \dot u_h\right>_{\Gamma_{f}} +  \left< g_N, \dot p_h \right>_{\Gamma_{f}}\\
&	   	   +\left< D_n \dot u_h , p_h \right>_{\Gamma_{\O}}
	   +  \left<  D_n u_h, \dot p_h \right>_{\Gamma_{\O}}
	   +    \left<  D_n \dot p_h, u_h\right>_{\Gamma_{\O}}
	   +  \left<  D_n p_h , \dot u_h\right>_{\Gamma_{\O}} \\
&-
	\beta h^{-1} 	\left<   \dot u_h, p_h \right>_{\Gamma_{\O}}
	-
	   \beta h^{-1} 	\left<    u_h, \dot p_h \right>_{\Gamma_{\O}} \\
 &- j(u_h,  \dot p_h) - j(\dot u_h, p_h).
\end{split}
\end{equation}
Combing \cref{domain-derivative-00-a}--\cref{domain-derivative-0-c} gives \cref{shape-derivative-formula-h1}.

\end{proof}

\begin{remark}
	The shape derivative is exact, however, due to the extra terms of the CutFEM formulation it becomes more complicated. We also observe that the field $\bftheta$ still has to be approximated in the finite element space (see section \ref{sec:beta} below).
\end{remark}


\subsection{CutFEM using boundary value correction}
In the classical shape derivative the function $u(x,t)$ and $p(x,t)$ are defined on the domain of $\O_t$. In this subsection, we instead define $u(x,t)$ on $\O$ (instead of $\O_t$) for $t$ small enough and include the effect of perturbations of the domain on the boundary through the weakly imposed boundary conditions, i.e., the boundary correction approach. The idea of perturbing boundary conditions to improve geometry approximation was first introduced in \cite{BDT72}. The extension to CutFEM was considered in \cite{BHL18}. For a recent discussion of the method interpreted as a singular Robin condition we refer to \cite{DGS20}. Similar ideas have already been exploited in the context of the standard Bernoulli problem, see \cite{BCT08}. 
Drawing on the ideas on boundary correction for the CutFEM method \cite{BHL18} we modify the weak formulation on the free boundary as follows:
\begin{equation}
\tilde{a}_h^t(w,v) = (\nabla w, \nabla v)_{\O} - \left<D_n w, v\right>_{\Gamma_{\O}}  - \left<D_n v, w \circ T_t\right>_{\Gamma_{\O}} + 
\beta h^{-1} (w \circ T_t,v \circ T_t)_{\Gamma_{\O}} ,
\end{equation}
and
\[
a_h^t(v,w) := \tilde{a}_h^t(v,w) + j(v,w).
\]
We note that the above weak formulation is consistent with the following:
\[
-\triangle u = f \in \O, \quad D_n u =g_N \mbox{ on } \Gamma_{f}, \quad \mbox{and} \quad u = 0 \mbox{ on } \Gamma_{\O_t}.
\]
Also note that the Dirichlet boundary condition that is originally weakly imposed on $\Gamma_{\O}$ is now weakly imposed on $\Gamma_{\O_t}$ through function composition.  

Now, considering the following variational problems:
finding $u_{h}(x,t) \in V_{h}(\O)$
 such that
\begin{equation}\label{u-t-h}
a_h^t(u_h(x,t),v) = (f,v)_\O+\left<g_N, v\right>_{\Gamma_{f}} 
\quad \forall v \, \in V_h(\O),
\end{equation}
and finding $p_{h}(x,t) \in V_h(\O)$ such that
\begin{equation}\label{p-t-h}
a_h^t(v, p_{h}(x,t)) =  h^{-1}\left<u_{h}(t)  -g_D , v\right>_{\Gamma_{f}} \quad \forall \,v \, \in V_h(\O).
\end{equation}
We  define the corresponding Lagrangian at pseudo-time $t$ with respect to $\bftheta$,
\begin{equation}
\begin{split}
\mathcal{L}_h^t(\O, u_{h}(t), p_{h}(t)) =& \dfrac{1}{2}h^{-1}\| g_D -u_h(x,t)\|^2_{\Gamma_{f}} -  a_h^t(u_h(x,t), p_h(x,t)) \\
&+(f,p_h(x,t))_\O+\left<g_N, p_h(x,t)\right>_{\Gamma_{f}}.
\end{split}
\end{equation}
\begin{remark}
It is easy to see that 
\[
\lim_{ t \to 0} \mathcal{L}_h^t(\O, u_h(t), p_h(t)) = \mathcal{L}_h(\O, u_h, p_h).
\]
\end{remark}
Finally, for a given $\bftheta$, we define the modified shape derivative  by
\begin{equation}\label{shape-derivative-2}
\tilde D_{\O, \bftheta} \mathcal{L}_h = \lim_{t \to 0} \dfrac{1}{t} 
\left( \mathcal{L}_h^t(\O, u_h(t), p_h(t)) -\mathcal{L}_h(\O, u_h(0), p_h(0)) \right),
\end{equation}
where $u_h(0) = u_h(\O), p_h = p_h(\O)$ are the solutions on $\O$ for \cref{u-h} and \cref{p-h}, respectively.

\begin{remark}
We note that contrary to the classical shape derivative here $u_h(x,t)$ and $p_h(x,t)$ are still defined on the fixed $\O$ and not on the perturbed domain $\O_t = \O + t \bftheta$. 
\end{remark}



\subsection{Shape derivative formula based on the boundary value correction}
In this subsection we derive the explicit formula of \cref{shape-derivative-2} in terms of $u_h(\O)$ and $p_h(\O)$.

Recall the pseudo-time derivative for $u_h$ and $p_h$:
\begin{equation}\label{time_derivative}
	u_h'(x) = \lim_{t \to 0} \dfrac{1}{t}(u_h(x,t) - u_h(x,0)),\;
	q_h'(x) = \lim_{t \to 0} \dfrac{1}{t}(q_h(x,t) - q_h(x,0)) \; \forall \,
	x \in \O.
\end{equation}

\begin{lemma}\label{lemma:sd-bv}
Let $u_h$ and $p_h$ be the solutions of \cref{u-h} and \cref{p-h}, respectively.
We have the following expression for the modified shape derivative defined in \cref{shape-derivative-2}:
\begin{equation} \label{shape-derivative-a}
\begin{split}
\tilde D_{\O, \bftheta} \mathcal{L}_h=
& \left<D_n p_h,  \nabla u_h \cdot \bftheta \right>_{\Gamma_{\O}} 
 -  \beta h^{-1} \left(  
	\left<  \nabla u_h \cdot \bftheta ,  p_h \right>_{\Gamma_{\O}} 
	+
	\left<  u_h ,  \nabla p_h \cdot \bftheta \right>_{\Gamma_{\O}}
	\right).
\end{split}
\end{equation}
\end{lemma}

\begin{proof}
By definition we have
\begin{equation}\label{a}
\begin{split}
	\tilde D_{\O, \bftheta} \mathcal{L}_h =& \lim_{t \to 0} \dfrac{1}{t} 
{\left(\mathcal{L}_h^t(\O, u_h(t), p_h(t) -\mathcal{L}_h(\O, u_h(0), p_h(0)) \right)}\\
=&   \lim_{t \to 0}  \dfrac{1}{2t} h^{-1} \left< u_h(t) - g_D)^2 -  (u_h(0) - g_D)^2 \right>_{\Gamma_{f}}\\
&- \lim_{t \to 0}  \dfrac{1}{t} \left( a_h^t( u_h(t), p_h(t)) -   a_h( u_h(0), p_h(0)) \right)\\
& +  \lim_{t \to 0}  \dfrac{1}{t} (f, p_h(t) - p_h(0))_\O +   
\lim_{t \to 0}  \dfrac{1}{t} \left<g_N, p_h(t) - p_h(0) \right>_{\Gamma_{f}}\\
&-\lim_{t \to 0}  \dfrac{1}{t} (j(u_h(t), p_h(t)) - j(u_h(0), p_h(0)))
\\
\triangleq& \sum_{i=1}^5\mathcal{A}_i.
\end{split}
\end{equation}

By a direct calculation and \cref{time_derivative} we have
\begin{alignat}{2}
\mathcal{A}_1 &= h^{-1}\left< u_h -g_D, u_h' \right>_{\Gamma_{f}},  
&\quad \mathcal{A}_3 &= (f, p'_h)_{\O}, \\
\mathcal{A}_4 &= \left< g_N, p_h' \right>_{\Gamma_{f}}, \quad
&  \mathcal{A}_5 &= -(j(u_h', p_h) - j(u_h, p_h')).
\end{alignat}
Expanding and regrouping terms in $a_h^t(\cdot)$ and $a_h(\cdot)$  gives 
\begin{equation}
\begin{split}
-\mathcal{A}_2 =& \lim_{t \to 0}  \dfrac{1}{t} \left( a_h^t( u_h(t), p_h(t)) -   a_h( u_h, p_h) \right)\\
=&
\lim_{t \to 0}	 \dfrac{1}{t} \left( (\nabla u_h(t), \nabla p_h(t))_{\O} -  (\nabla u_h(0), \nabla p_h(0))_{\O}  \right)\\
&
-\lim_{t \to 0}	 \dfrac{1}{t} \left( \left<D_n u_h(t), p_h(t)\right>_{\Gamma_{\O}} 
-  \left<D_n u_h(0), p_h(0)\right>_{\Gamma_{\O}} \right)\\
&-\lim_{t \to 0}	 \dfrac{1}{t} \left( \left<D_n p_h(t), u_h(t) \circ T_t\right>_{\Gamma_{\O}}  
		-\left<D_n p_h(0), u_h(0)\right>_{\Gamma_{\O}}  \right)\\
& + \lim_{t \to 0} \dfrac{1}{t} \beta h^{-1} \left(  \left<u_h(t) \circ T_t, p_h(t) \circ T_t\right>_{\Gamma_{\O}} - 
		 \left< u_h(0), p_h(0)\right>_{\Gamma_{\O}} \right).
\end{split}
\end{equation}
Applying the product rule, Taylor expansion and neglecting the higher order terms gives
\begin{equation}
\begin{split}
-\mathcal{A}_2 =& \lim_{t \to 0}  \dfrac{1}{t} \left( a_h^t( u_h(t), p_h(t)) -   a_h( u_h, p_h) \right)\\
=&
\left( (\nabla u_h', p_h)_{\O} +  (\nabla u_h, \nabla p_h')_{\O}  \right)
-\left( \left<D_n u'_h, p_h \right>_{\Gamma_{\O}} -  \left<D_n u_h, p_h'\right>_{\Gamma_{\O}} \right)\\
&-\lim_{t \to 0}	 \dfrac{1}{t} \left( \left<D_n p_h(t), u_h(t) + t \nabla u_h(t) \cdot \bftheta \right>_{\Gamma_{\O}}  
		-\left<D_n p_h, u_h\right>_{\Gamma_{\O}}  \right)\\
& + \lim_{t \to 0} \dfrac{1}{t} \beta h^{-1} \left(  
	\left<  u_h(t) + t \nabla u_h(t) \cdot \bftheta ,  p_h(t) + t \nabla p_h(t) \cdot \bftheta \right>_{\Gamma_{\O}} - 
	\left< u_h , p_h\right>_{\Gamma_{\O}} \right)\\
=&
\left( (\nabla u_h', \nabla p_h)_{\O} +  (\nabla u_h, \nabla p_h')_{\O}  \right)
-\left( \left<D_n u'_h, p_h \right>_{\Gamma_{\O}} -  \left<D_n u_h, p_h'\right>_{\Gamma_{\O}} \right)\\
&-	 \left( \left<D_n p_h', u_h \right>_{\Gamma_{\O}}  
		+ \left<D_n p_h, u_h'\right>_{\Gamma_{\O}}  
		+\left<D_n p_h,  \nabla u_h \cdot \bftheta \right>_{\Gamma_{\O}} 
		\right)\\
& +  \beta h^{-1} \left(  
	\left<  u_h' ,  p_h \right>_{\Gamma_{\O}} +\left< u_h , p'_h\right>_{\Gamma_{\O}} 
	+
	\left<  \nabla u_h \cdot \bftheta ,  p_h \right>_{\Gamma_{\O}} 
	+
	\left<  u_h ,  \nabla p_h \cdot \bftheta \right>_{\Gamma_{\O}} 
	\right).
\end{split}
\end{equation}
Note that $u_h' , p_h' \in V_h$. 
By \cref{u-h} and \cref{p-h} we have
\begin{equation}
\begin{split}
&a_h(p_h, u'_h) \\
=& (\nabla p_h, \nabla u_h')_{\O} - \left<D_n p_h, u_h' \right>_{\Gamma_{\O}}  - \left<D_n u_h', p_h\right>_{\Gamma_{\O}} + 
\beta h^{-1} \left<p_h,u_h'\right>_{\Gamma_{\O}} + j(p_h, u_h') \\
=& h^{-1} \left< u_h - g_D, u_h'\right>_{\Gamma_{f}}
\end{split}
\end{equation}
and
\begin{equation}\label{b}
\begin{split}
&a_h(u_h, p'_h) \\
=& (\nabla u_h, \nabla p_h')_{\O} - \left<D_n u_h, p_h' \right>_{\Gamma_{\O}}  - (D_n p_h', u_h)_{\Gamma_{\O}} + 
\beta h^{-1} (u_h,p_h')_{\Gamma_{\O}} + j(u_h, p_h') \\
=&(f, u_h')_\O +  \left< g_N, p_h'\right>_{\Gamma_{f}}
\end{split}
\end{equation}
Combining \cref{a}--\cref{b} gives \cref{shape-derivative-a}. This completes the proof of the lemma.
\end{proof}

\begin{rem}
Applying the Taylor expansion and omitting the higher order terms gives
\begin{equation} \label{shape-derivative-b}
\begin{split}
a_h^t(w,v) &\approx (\nabla w, \nabla v)_{\O} - \left<D_n w, v\right>_{\Gamma_{\O}}  - (D_n v, w)_{\Gamma_{\O}} + 
\beta h^{-1} (w, v)_{\Gamma_{\O}} \\
&- t \left( (D_n v, \nabla w \cdot \bftheta)_{\Gamma_{\O}}
+
\beta h^{-1} (\nabla w \cdot \bftheta, v)_{\Gamma_{\O}}
+
\beta h^{-1} (\nabla v \cdot \bftheta, w)_{\Gamma_{\O}} \right).
\end{split}
\end{equation}
Taking the derivative with respect to $t$ also gives  \cref{shape-derivative-a}.
\end{rem}

\begin{rem}
We note that here the modified shape derivative $\tilde D_{\O, \bftheta}$ is exact.
However, comparing to \cref{shape-derivative-formula-h1}  the formula in \cref{shape-derivative-a}  is much more simple. Furthermore, since the shape derivative has the surface form and it is exact, it would be an interesting alternative when an explicit parametric approach for the surface representation is used.
\end{rem}



\section{Optimization algorithms}\label{sec:Optimization-Algorithms}
The objective is now to find the vector field $\bftheta: \hat \O \rightarrow \hat \O$  such that $J(\O)$ decreases the fastest along that direction. We seek through solving the following constrained minimization problem:
starting from the domain $\O$ with free boundary $\Gamma_{\O}$ we wish to find the steepest descent vector field
$\bfbeta \in W(\hat\O, \mathbb{R}^d)$ such that 
\begin{equation}\label{minimization2}
\bfbeta = \underset{\substack{\|\bftheta\|_{H^1(\hat \O)} =1, \\ \bftheta = 0 \mbox{ on } \Gamma_{f}.}}{\mbox{argmin}}
D_{\O, \bftheta} \mathcal{L}_h.
\end{equation}
Define the corresponding Lagrangian 
\[
	\mathcal{K}(\bftheta, \lambda) = D_{\O, \bftheta} \mathcal{L}_h + 
	\lambda(\| \bftheta\|_{H^1(\hat \O)}^2-1),
\]
and taking the derivative with respect to $\lambda$ gives the constrain condition. 
From remark 4.1 in \cite{BEHLL17}, an equivalent formulation of \cref{minimization2} renders
to find $\tilde\bfbeta \in H_0^1(\hat \O)^d$ such that
\begin{equation}\label{velocity-continuous-formula}
	(\tilde\bfbeta, \bftheta)_{H^1(\hat \O)} = -D_{\O, \bftheta}\mathcal{L}_h   
	\quad \forall \, \bftheta \in H_{0}^1( \hat \O)^d.
\end{equation}
where $\tilde\bfbeta= 2\lambda \bfbeta$ and $\lambda = \dfrac{\| \tilde\bfbeta\|_{H^1(\hat \O)}}{2}$. Then it is easy to see
 that by taking $\bftheta = \bfbeta$ 
\[
	 D_{\O, \bfbeta}\mathcal{L}_h = - (\tilde\bfbeta, \bfbeta)_{H^1(\hat \O)} 
	=-\|\tilde\bfbeta\|_{H_{0}^1(\hat \O)} <0
\]
which guarantees that $\bfbeta$ is a descent direction.

The following Hadamard Lemma indicates that under certain regularity the variational problem \cref{velocity-continuous-formula} is equivalent to an interface problem.
\begin{lem}[Hadamard]
If $\mathcal{L}(\cdot)$ is shape differentiable at every element $\O$ of class $C^k, \O \subset \hat \O$.
Furthermore, assume that $\partial \O$ is of class $C^{k-1}$. Then there exists a scalar function
$\mathcal{G}(\Gamma_{\O}) \subset \mathcal{D}^{-k}(\Gamma_{\O})$ such that
\begin{equation}
	D_{\O, \bftheta} \mathcal{L}(\O) = \int_{\Gamma_{\O}} \mathcal{G}\bftheta \cdot \bn \,ds.
\end{equation}
\end{lem}
It therefore follows from the above lemma and \cref{velocity-continuous-formula} that
\begin{equation}\label{hadamond}
	(\nabla \tilde\bfbeta, \nabla  \bftheta)_{\O} + (\tilde \bfbeta, \bftheta)_{\O} = 
	- \int_{\Gamma_{\O}} \mathcal{G} \bftheta \cdot \bfn \,ds.
\end{equation}
Equation \cref{hadamond} indicates that, in strong form, we need to solve the following interface problem for $\tilde \bfbeta$,
\begin{alignat}{3}
	-\triangle \tilde \bfbeta +  \tilde\bfbeta &=0   &\quad & \text{in $\hat \O$},
\\
\jump{D_n  \tilde\bfbeta}|_{\Gamma_{\O}} &= - \mathcal{G} 	& \quad & \text{on $\Gamma_{\O} $},	 
\\	
\jump{ \tilde\bfbeta} &=0  &\quad & \text{on $\Gamma_{\O} $},
\\
 \tilde\bfbeta &= 0  &\quad & \text{on $\partial  \hat \O$}. 	
\end{alignat}
Given that $\Gamma_{\O}$ is smooth and $\mathcal{G} \in H^{1/2}(\Gamma_\O)$,
we also have the following regularity estimate:
\begin{equation}
\begin{split}
\| \tilde \bfbeta\|_{H^1( \hat \O)} + \| \tilde\bfbeta\|_{H^2( \hat \O \setminus \Gamma_\O)} \lesssim \| \mathcal{G}\|_{H^{1/2}(\Gamma_{\O})},
\end{split}
\end{equation}
(see \cite{CZ98}) and hence $ \tilde\bfbeta \in  {H^1(\hat \O)} \cap {H^2(\hat \O \setminus \Gamma_{\O})}$.


\subsection{Approximation of the steepest descent velocity using CutFem}\label{sec:beta}

To obtain a numerical approximation for the steepest descent velocity, we also use the CutFEM for interface problem \cite{HaHa02} on a single mesh. We first define the finite element spaces.
Given a closed $d-1$ manifold $\Gamma \subset \hat\O$, define $\O_\Gamma^+ \subset \hat \O$ be the intersection of the domain enclosed by $\Gamma$ and $\Gamma_f$
and define $\O_\Gamma^- = \hat \O \setminus \O_\Gamma^+$.
Also define the finite element space $V_h^+$ and $V_h^-$ by
\[
V_h^+ =\{ v^+ \in H^1(\O^+) : v_1|_{K} \in P^1(K) \quad \forall K \cap \O^+ \neq \emptyset
\},
\]
and
\[
V_h^- =\{ v^- \in H^1(\O^-) : v^-|_{K} \in P^1(K) \quad \forall K \cap \O^- \neq \emptyset
\}.
\]
Note that both $V_h^+$ and $V_h^-$ are defined on ``cut" elements $K \in \cT_h$ such that $K \cap \Gamma \neq \emptyset$.
The finite element solution for $\bfbeta$ is then to find $\bfbeta_h \in V_h^+ \times V_h^-$  such that
\begin{equation}\label{velocity-approx}
\begin{split}
	b_0( \bfbeta_h, \bfv) + j(\bfbeta_h, \bfv)= l(\bfv)
	\quad \forall \, \bfv \in V_h^+ \times V_h^-
\end{split}
\end{equation}
where
\begin{equation}
\begin{split}
	b_0( \bfbeta, \bfv) &= (\nabla \bfbeta^+, \nabla \bfv^+)_{\O^+} 
	+ (\nabla \bfbeta^-, \nabla \bfv^-)_{\O^-}
	- \left<\{D_n \bfbeta\}, \jump{\bfv} \right>_{\Gamma_\O} 
	- \left<D_n \bfbeta, \bfv \right>_{\Gamma_f}   \\
	&- \left<\{D_n \bfv\}, \jump{\bfbeta}\right>_{\Gamma_\O}
	+ \beta_1 h^{-1} \left< \jump{\bfbeta}, \jump{\bfv} \right>_{\Gamma_\O}
	- \left<D_n \bfv, \bfbeta \right>_{\Gamma_f}
	+ \beta_2 h^{-1} \left< \bfbeta, \bfv \right>_{\Gamma_f}
\end{split}
\end{equation}
\begin{equation}
\begin{split}
	j(\bfbeta, \bfv)=\gamma h \left( \sum_{F \in \cE_I^+} \int_F \jump{D_n \bfbeta^+} \jump{D_n \bfv^+}
	+
	\sum_{F \in \cE_I^-} \int_F \jump{D_n \bfbeta^-} \jump{D_n \bfv^-} \right)
\end{split}
\end{equation}
and
\begin{equation}\label{rhs-shape-derivative}
l(\bfv) = -D_{\O, \bfv} \mathcal{L}_h \quad \mbox{ or } \quad
l(\bfv) = -\tilde D_{\O, \bfv} \mathcal{L}_h.
\end{equation}
Here $
	\cE_I^\pm = \{ F\,: \,F \in \cE_I, K \cap \O_\Gamma^\pm \neq \emptyset,  K' \cap \O_\Gamma^\pm \neq \emptyset \mbox{ where } K \cap K' = F\}.
$

\subsection{Level set update}
With the steepest direction on hand, we now aim to update the free boundary.
By introducing the pseudo-time, we aim to find the level set function $\phi(x+ t \bfbeta(x),t)$ for some given $\bfbeta$ such that 
\[ \phi(x+ t \bfbeta(x),t) = \phi(x,0) \quad \forall \, t \mbox{ and } \forall \, x \in \hat \O.
\]
Taking the derivative with respect to $t$ gives that
\[
	\nabla_x \phi \cdot \bfbeta + \dfrac{\partial \phi}{\partial t}  =0,
\]
which yields a Hamilton-Jacobi equation, if the nonlinear dependence of $\bfbeta$ on the optimization is accounted for. However for fixed vector field $\bfbeta$ this is simply an advection problem with a non-solenoidal transport field.
\begin{remark}
	Note that the level set function chosen at the initial stage is the distance function. However, after evolution steps the updated level set function no longer holds the property of a distance function. This could cause potential problems, for accuracy if the magnitude of the gradient locally becomes very small and for the stability of the numerical scheme if the gradient becomes very large. It is well known that the issue can be resolved by redefining $\phi$ regularly as the distance function while keeping the interface position fixed. In the numerical examples presented herein we did not notice any need for re-distancing, since an advection stable scheme was used to propagate the interface.
\end{remark}
To approximate the Hamilton-Jacobi equation, we use Crank-Nicolson scheme in time combining with gradient penalty stabilization in space for the advection problem \cite{BF09,Bu11}. We keep the same background mesh for the transport of the level set function. 

For each $\O$, let $T =r*\dfrac{J(\O)}{\|\bfbeta_h\|_{H^1(\hat \O)}}$, where $r$ is the learning rate. 
First divide $[0,T]$ into $N$ equal length steps and let $\delta t = T/{N}$ and $t_i = i \delta_t$ for $ i =0, \cdots N$. 
 Denote by $\phi_h^n = \phi_h(t_n)$. Given the initial level set $\phi_h^0$, find $\phi_h^n \in V_h$ for $n =1, \cdots, N$  such that
\begin{equation}\label{H-J-approx}
	\left( \dfrac{\phi_h^n - \phi_h^{n-1}}{\delta t},w \right)_{D} 
	+ \left(\bfbeta_h \cdot \nabla \dfrac{\phi_h^n + \phi_h^{n-1}}{2},w \right)_{D}
	+ r_h\left( \dfrac{\phi_h^n + \phi_h^{n-1}}{2} ,w\right)
\end{equation}
where
\[
r_h(v ,w) = \sum_{F \in \cE_I} \gamma_2 h^2 \int_F \jump{D_n v} \jump{D_n w} \,ds
\]
with $\gamma_2>0$ is a parameter and $\cE_I$ is the set of all interior facets in $\cT_0$.
In the numerics, we chose $r=1.0$ or $0.5$, $N =10$ and $\gamma_2 = 1.0$.

Below we summarize the algorithm.
\begin{algorithm}[h!] 
\caption{Bernoulli Free Boundary Identification}
\label{alg:buildtree}
\begin{algorithmic}
\STATE{Input an initial level set $\phi_h$ and specify the tolerance.}
\WHILE{ J $>$ tolerance }
	\STATE{Compute the primal solution $u_{h}$ by \cref{u-h} and the dual solution $p_h$ by \cref{p-h}. }
	\STATE{Compute $J$.}
	\STATE{Compute the velocity $\bfbeta_h$ by \cref{velocity-approx}.}
	\STATE {Compute $T =r*\dfrac{J}{\|\bfbeta_h\|_{H^1(\hat \O)}}$. ($0 < r <1$ is the learning rate).}
	\STATE{Normalize $\bfbeta_h$.}
	\STATE{Compute $\phi_h(x,T)$ by \cref{H-J-approx}.}
	\STATE{Set $\phi_h(x,0) = \phi_h(x,T)$}.
\ENDWHILE
\end{algorithmic}
\end{algorithm}


\section{Numerical experiments}\label{sec:Numerical Experiments}
In the numerical experiments we mainly aim to compare the performances of the three different shape derivatives, i.e., the classical shape derivative (SD) given in \cref{shape-derivative-formula} obtained based on the first optimize then discretize approach, \cref{shape-derivative-formula-h1} obtained based on the first discretize then optimize approach,  and the  \cref{shape-derivative-a} obtained based on the boundary correction approach. For simplicity, in this section we refer the three shape derivatives as the \textit{continuous SD, discrete SD} and \textit{boundary SD}.

For the CutFEM method a regular fixed background mesh is used. For all numerical experiments in this paper we will use the unit square domain as the background domain, i.e.,  $\hat \O = [0,1]^2$ and the background mesh is a uniform $100 \times 100$ mesh. The penalty parameters in \cref{a-h} are chosen as $\gamma =0.1$ and $\beta=10$. And in \cref{velocity-approx}, the parameters are chosen such that $\beta_1 = \beta_2 = 10.0$ and $\gamma = 1.0$.

\begin{example}[Circle]\label{ex:circle} 
We recall the problem:
\begin{equation}
\begin{split}
 &-\triangle u = f \quad \mbox{in} \;\Omega, \\
 &u = 0  \quad \mbox{on} \;\Gamma_{\O},\\
 &u = g_D, \; D_n u = g_N \quad  \mbox{on}\;  \Gamma_{f}.
\end{split}
\end{equation}
For this example, the free boundary $\Gamma_{\O}$ is the circle with radius $r = 1/4$ and center being $(0.5, 0.5)$.
\end{example}
We chose to use the following data set:
\begin{equation}\label{data-set}
	u = 4r -1 \mbox{ and } \; f = -4/r.
\end{equation}
We note that the data set is not unique and indeed there are infinitely many choices. 

We start with the following initial level set:
\[
	\phi(r, \theta) = -r + 1/8,
\]
which is a smaller circle with the same center as the true interface (see the inner most  red circle in  \cref{fig:level-ser1-1}).



\begin{figure}
    \centering
    \subfloat[]{\includegraphics[width=0.40\textwidth]
    {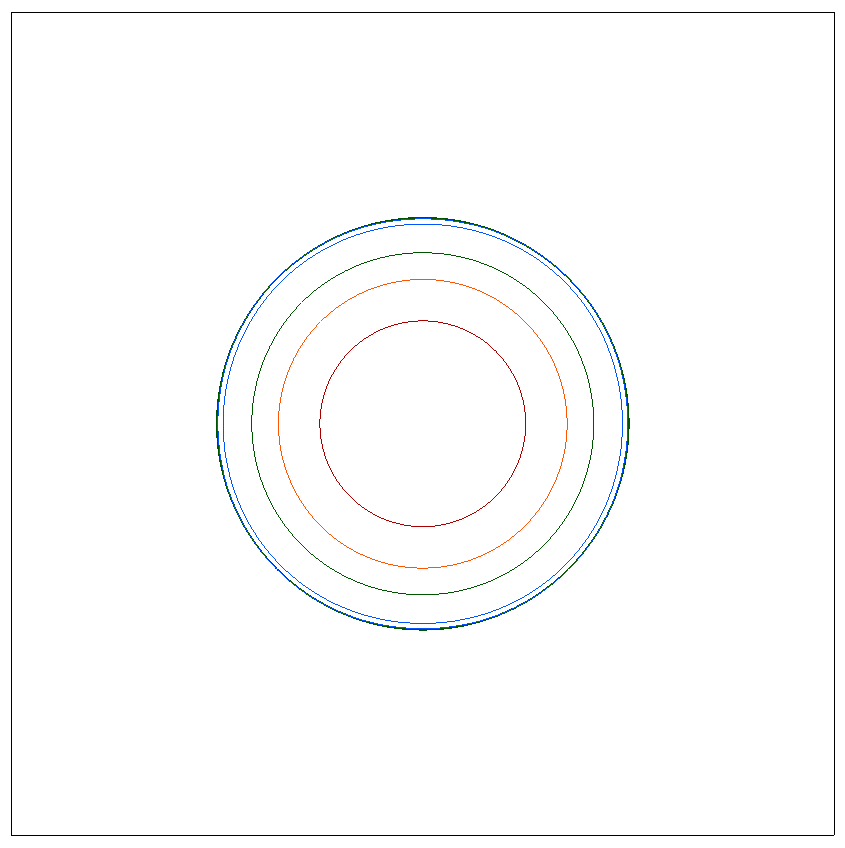} 
        \label{fig:level-ser1-1}}
    \hfill
        \subfloat[]{\includegraphics[width=0.52\textwidth]{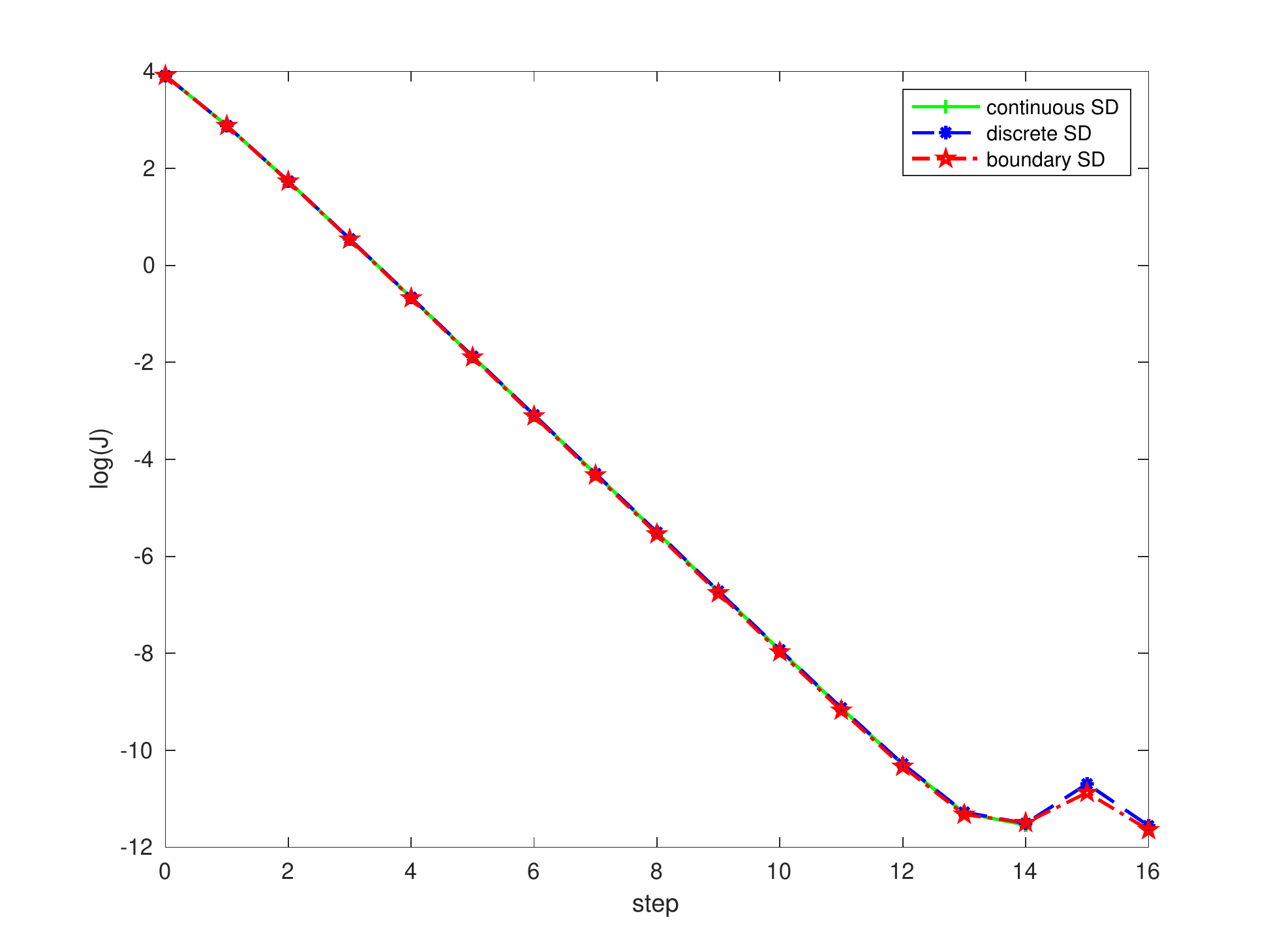}
       \label{fig:Ex1-innerball}}
\caption{\cref{ex:circle}: {\upshape{(a):}} level sets at steps $0, 1,2, 5$ and $10$;  {\upshape{(b):}} the comparison of residual evolution}
\end{figure}

The stopping criteria is set such that $J(\Omega)\le 1E-5$. 
It takes $14$, $16$ and $16$ steps, respectively, using the continuous SD, discrete SD and boundary SD
to reach the stopping criteria. In this case, the performances between those three shape derivatives are almost identical. \cref{fig:level-ser1-1} shows the level sets at steps $0, 1, 2, 5$ and $10$ (from the inner most the to outer most circles) for all three shape derivatives. The level set at step $0$ is the initial guess of level set. At step 10, the computed level set almost coincides with the true level set. \cref{fig:Ex1-innerball} shows the decreasing log rate of the residuals. In this case all residuals converge at a uniform rate.

We now test with an initial level set as an ellipse (see the red curve in \cref{fig:Ex1-ellipse-0}):
\[
	\phi(x, y) = - \dfrac{(x -0.5)^2}{c_1^2} - \dfrac{(x -0.5)^2}{c_2^2} +1, \;\mbox{where} \, c_1 = 3/8, \mbox{ and } c_2 = 1/8.
\]
\cref{fig:Ex1-ellipse-0} - \cref{fig:Ex1-ellipse-50} show the obtained level sets at steps $0$, $5, 10, 20$ and $50$  
using the continuous SD (green), discrete SD (blue) and boundary SD (red).
With the same stopping criteria that $J \le 1E-5$, it takes $169$ , $155$, and $123$ steps respectively for the continuous SD, discrete SD and boundary SD. We note that in this case using the boundary SD shows slightly better performance.

In \cref{fig:Ex1-ellipse-residual-100} we compare the residual evolution for the first $100$ steps. We note that there are two different convergence patterns for all shape derivatives: in the first $20$ steps the residual is decreasing at a uniform fast rate and afterwards evolves at a much slower rate. 

\begin{figure}
    \centering
    \subfloat[]{\includegraphics[width=0.45\textwidth]
    {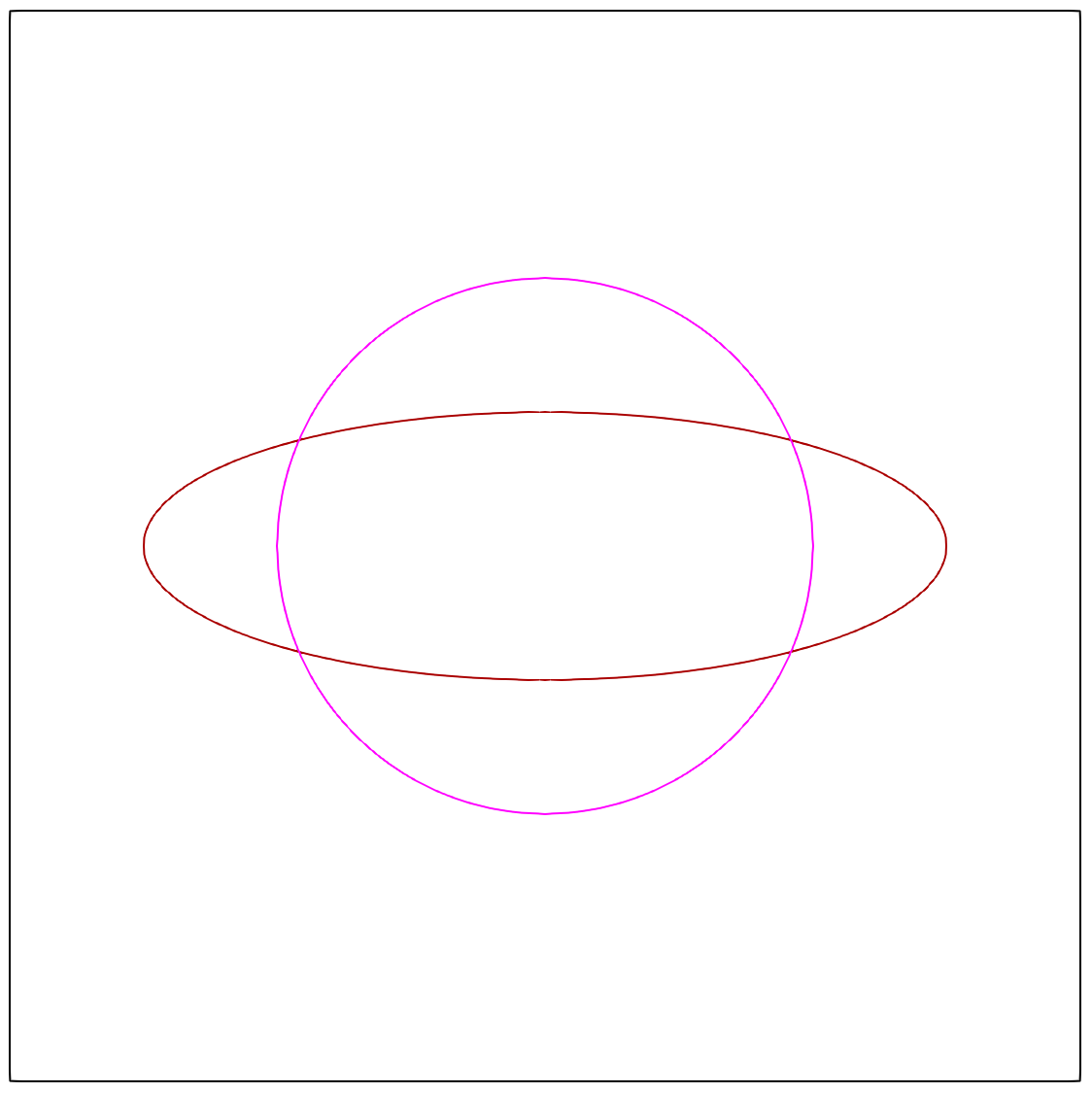} 
        \label{fig:Ex1-ellipse-0}}
    \hfill
        \subfloat[]{\includegraphics[width=0.45\textwidth]{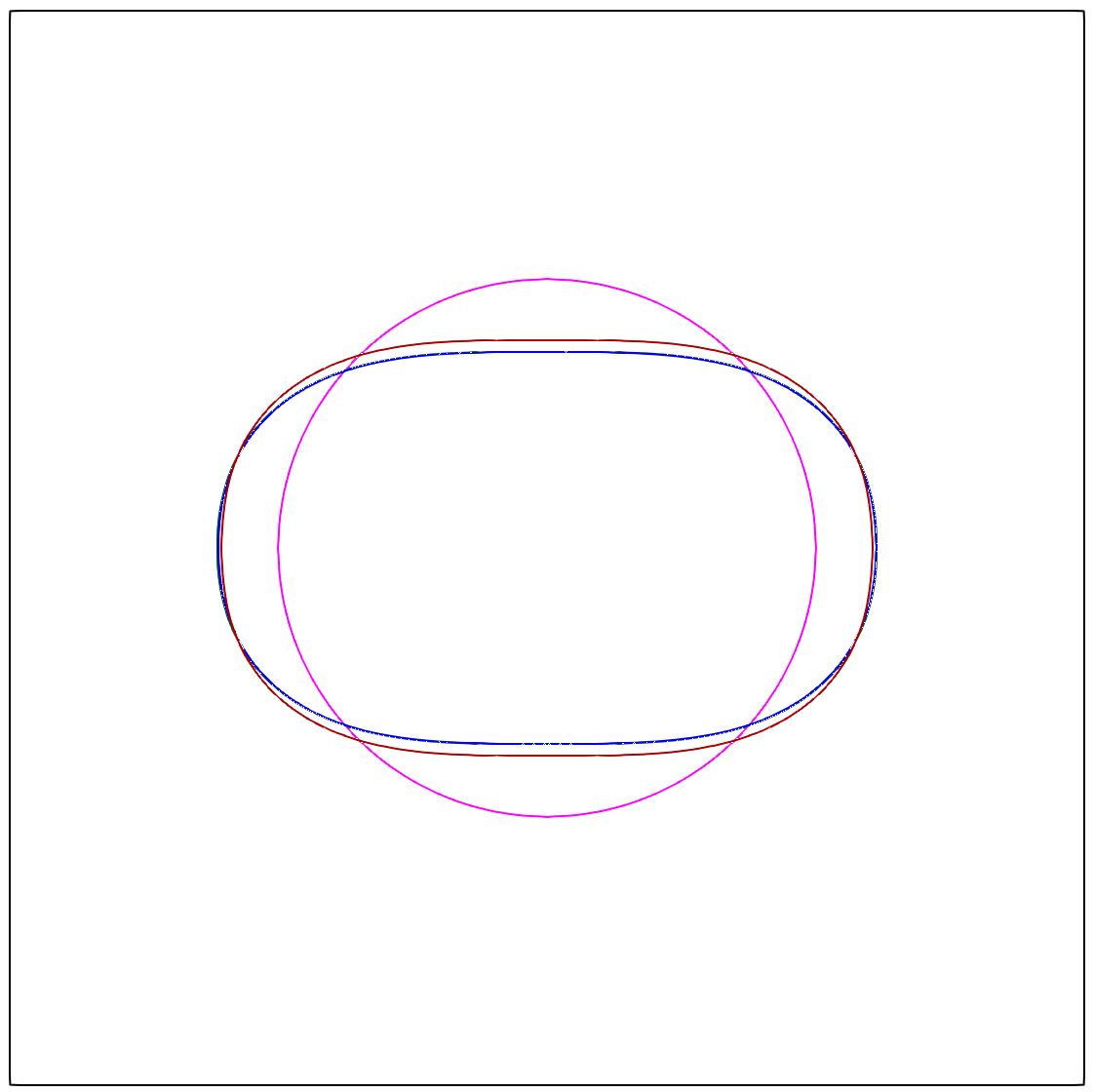}
       \label{fig:Ex1-ellipse-5}}
          \hfill
        \subfloat[]{\includegraphics[width=0.45\textwidth]
    {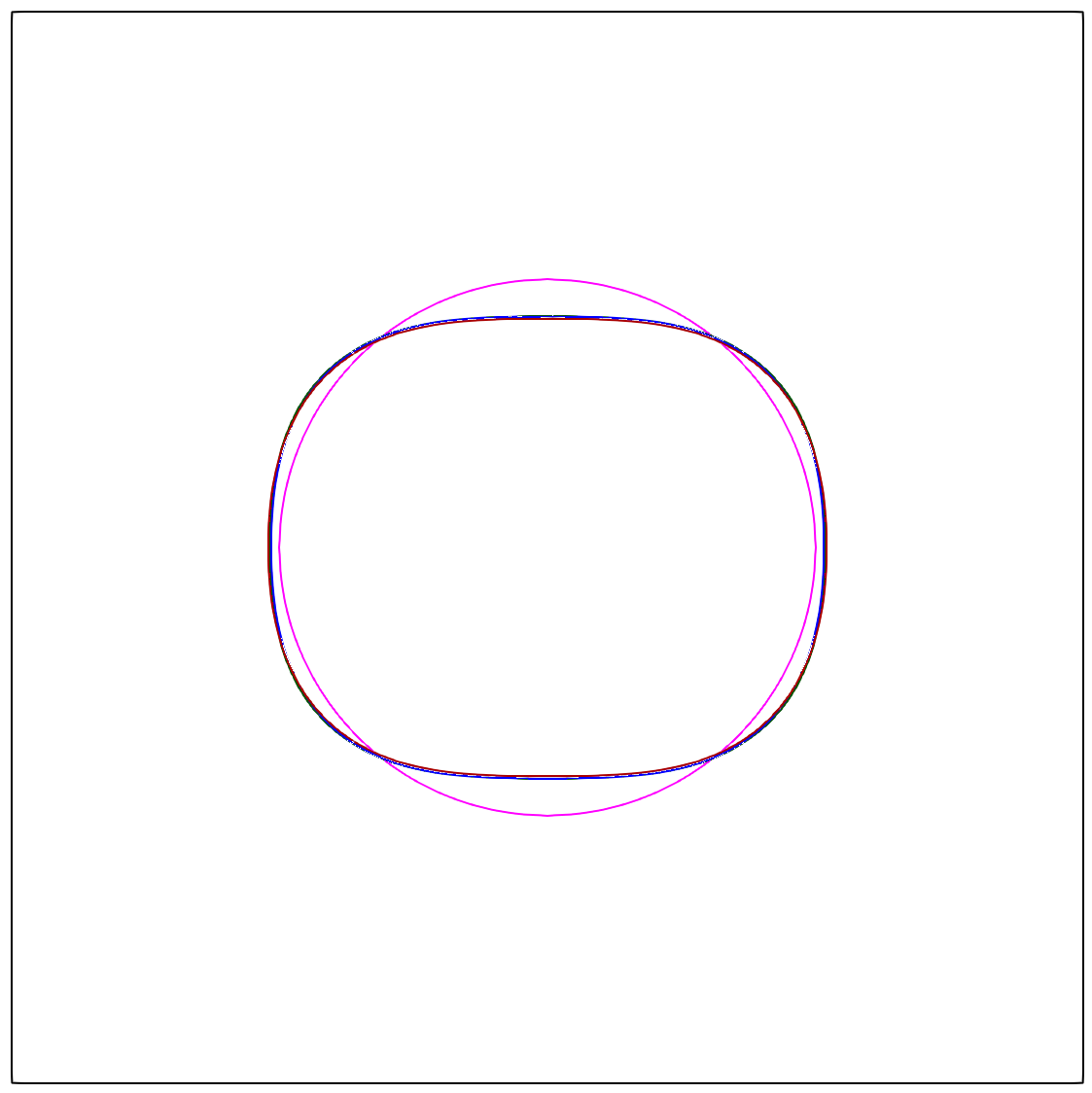} 
        \label{fig:Ex1-ellipse-10}}
    \hfill
    \subfloat[]{\includegraphics[width=0.45\textwidth]
    {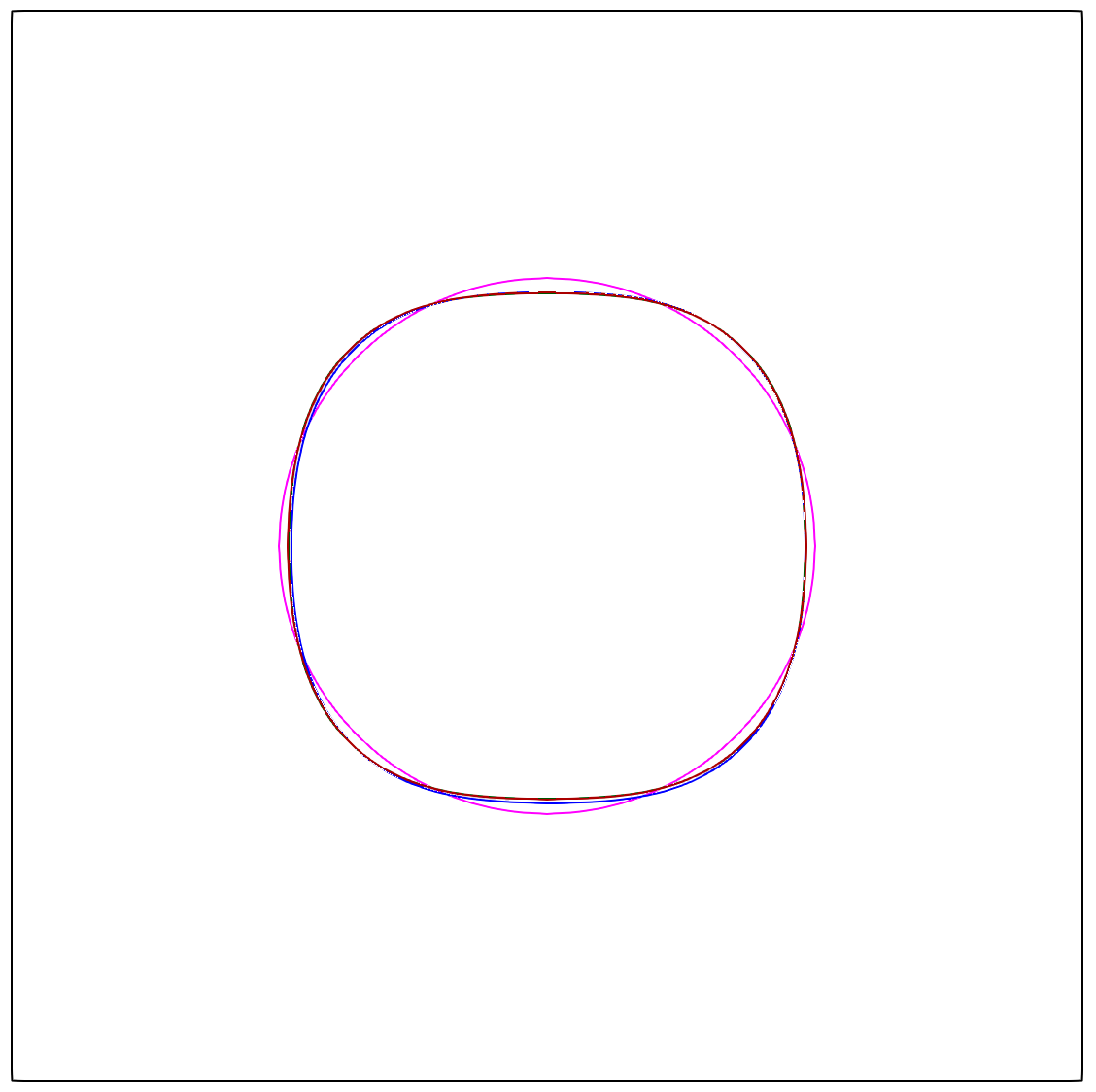} 
        \label{fig:Ex1-ellipse-20}}
    \hfill
            \subfloat[]{\includegraphics[width=0.45\textwidth]
    {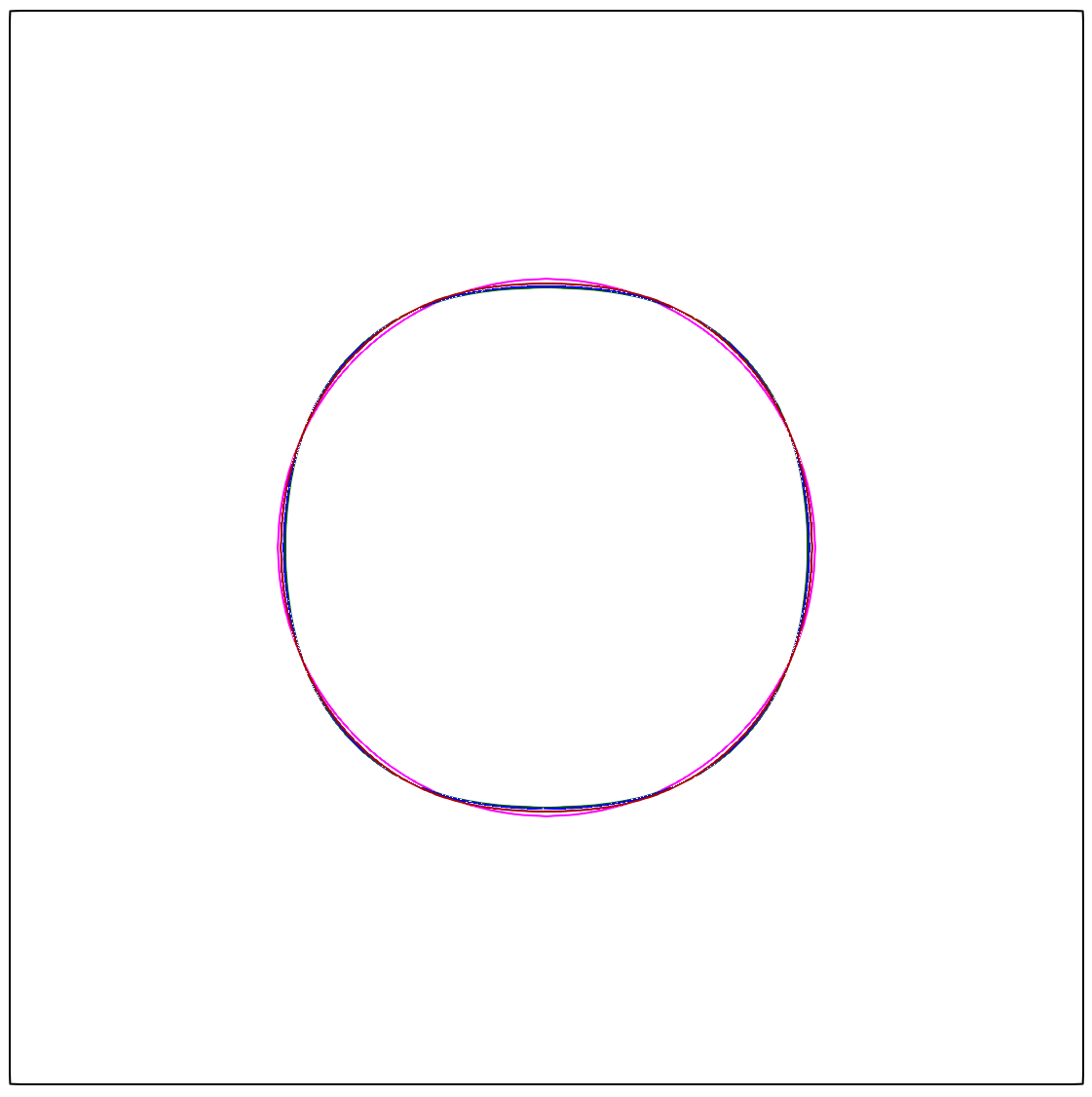} 
        \label{fig:Ex1-ellipse-50}}
    \hfill
        \subfloat[]{\includegraphics[width=0.45\textwidth]
    {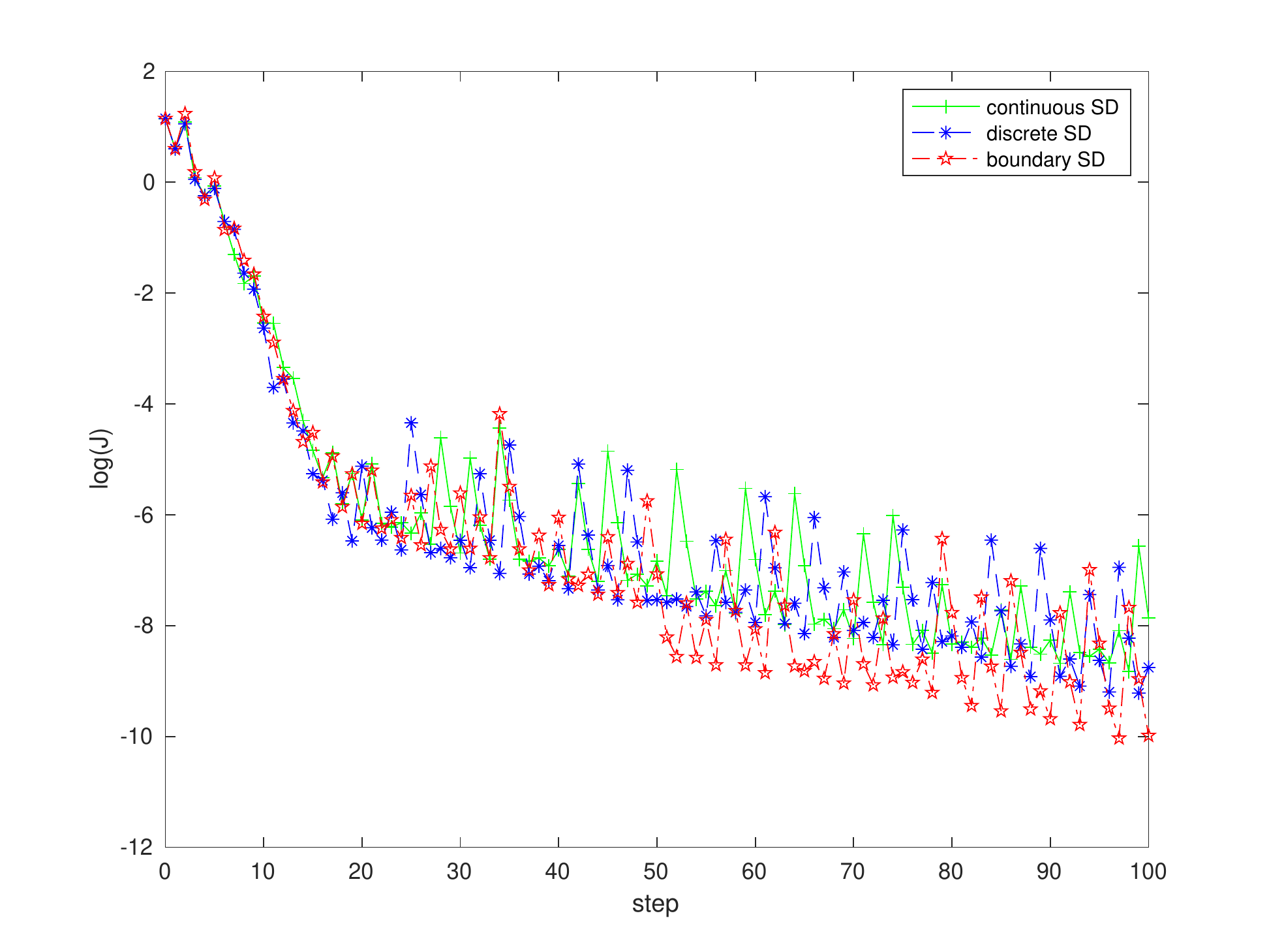} 
        \label{fig:Ex1-ellipse-residual-100}}
\caption{\cref{ex:circle}: level sets at steps $0${\upshape{(a)}}, $5${\upshape{(b)}}, $10${\upshape{(c)}} ,$20${\upshape{(d)}}, $50${\upshape{(e)}} and the comparison of residual evolution {\upshape{(d)}}}
\end{figure}

If the initial guess is not properly chosen, the iterative procedure might need much more steps to converge. Moreover, since we are using a gradient method, the minimum obtained is a local minimum.
We also note that it is natural that the residual oscillates when the pseudo time step is fixed. Also observe that although the SD may be exact for the discrete formulation it is not necessarily in the finite element space and must nevertheless be approximated.

\begin{example}[Ellipse]\label{ex:2}
For this example, the free boundary $\Gamma_\O$ is an ellipse (see the magenta curve in \cref{fig:Ex2-circle-0}) with the following level set representation:
\[
	\phi(x,y) = -16(x-0.5)^2 - 64(y - 0.5)^2 + 1.
\]
We chose to use the data set such that $f=0$, $g_N = (\sin(x+y), \cos(x+y)) \cdot \bn$ and $g_D$ is obtained by solving the forward problem on a $500 \times 500$ mesh.
\end{example}

We start with an initial level set of a circle (see the red circle in \cref{fig:Ex2-circle-0}):
\[
	\phi(x,y) = -\sqrt{(x-0.6)^2 + (y - 0.4)^2} + 1/6,
\]
which is partially intersected with the true interface. 
\cref{fig:Ex2-circle-0}--\cref{fig:Ex2-circle-120}  show the obtained level sets at steps $0, 5, 10, 50$ and $120$  
using the continuous SD (green), discrete SD (blue) and the boundary SD (red).
With the stopping criteria that $J \le 1E-5$, it takes $120$ , $154$, and $146$ steps respectively for the continuous SD, discrete SD and boundary SD. 
We again observe  that the level sets and the residual revolution of three methods are all very similar. 
However, the number of steps that it takes to reach the stopping criteria could differ quite a lot due to its slow convergence rate and oscillating character of the costl functional. 

\begin{figure}
    \centering
      \subfloat[]{\includegraphics[width=0.45\textwidth]
    {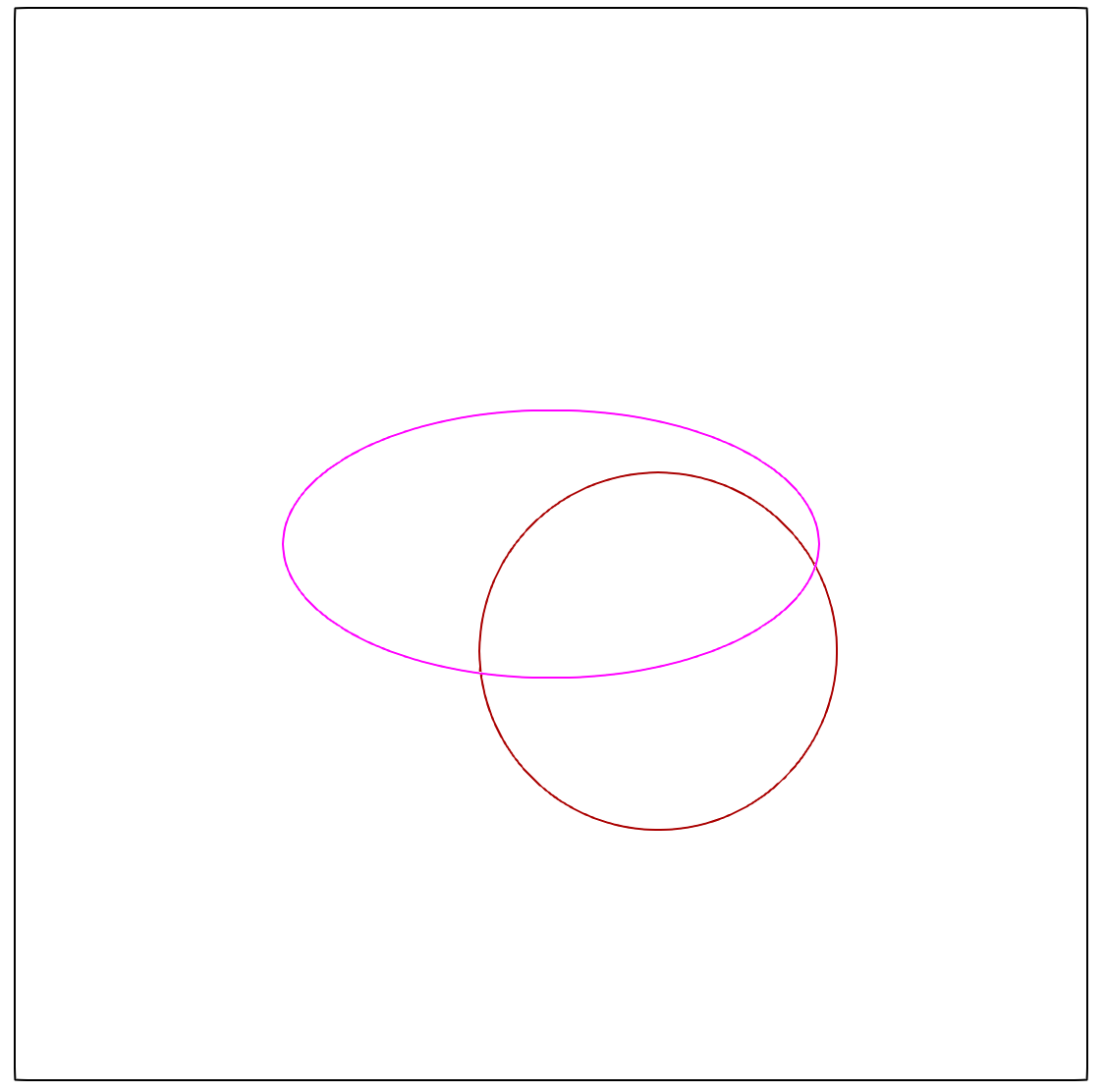} 
        \label{fig:Ex2-circle-0}}
    \hfill
    \subfloat[]{\includegraphics[width=0.45\textwidth]
    {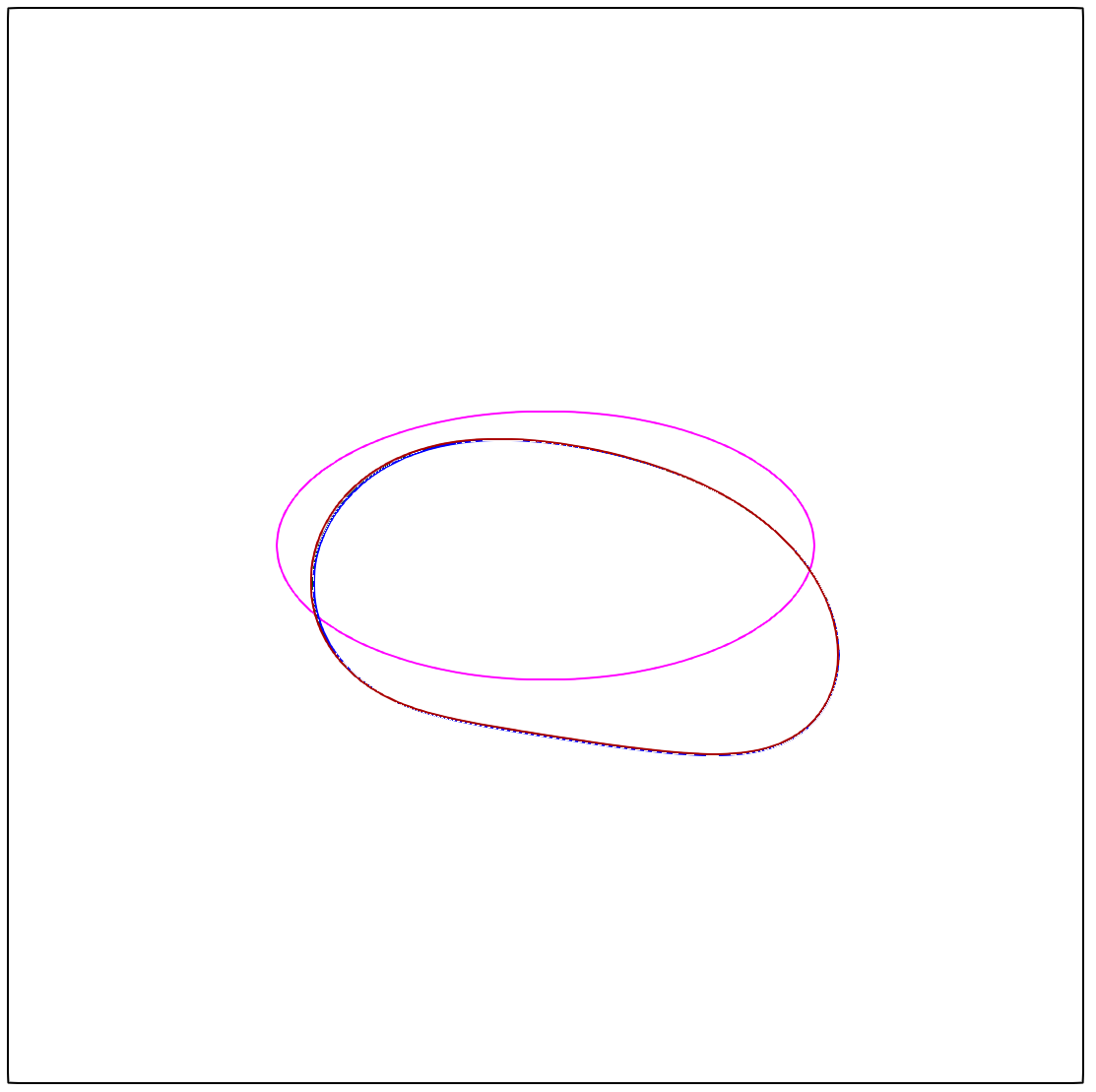} 
        \label{fig:Ex2-circle-5}}
    \hfill
        \subfloat[]{\includegraphics[width=0.45\textwidth]{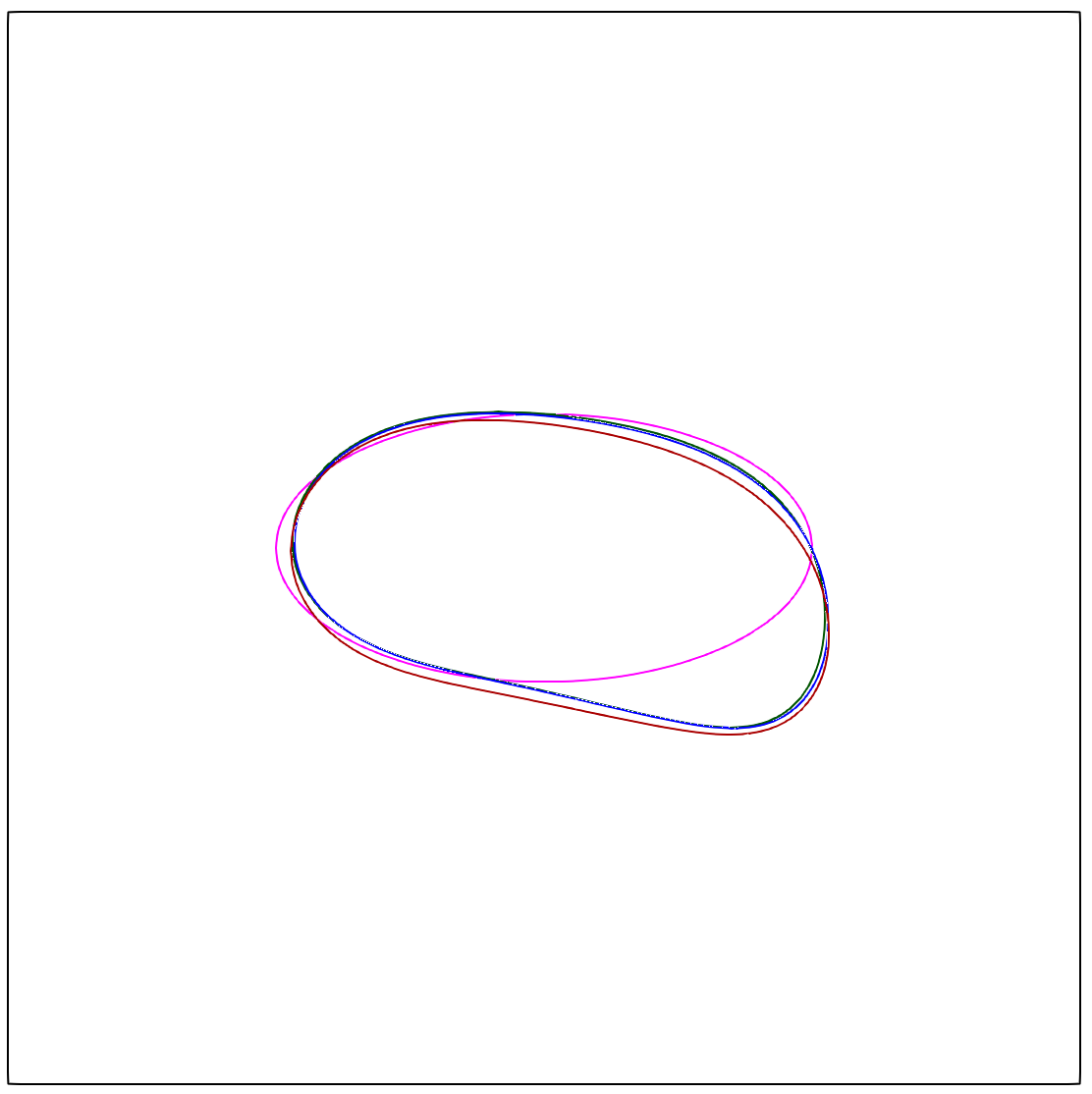}
       \label{fig:Ex2-circle-10}}
       \hfill
        \subfloat[]{\includegraphics[width=0.45\textwidth]
    {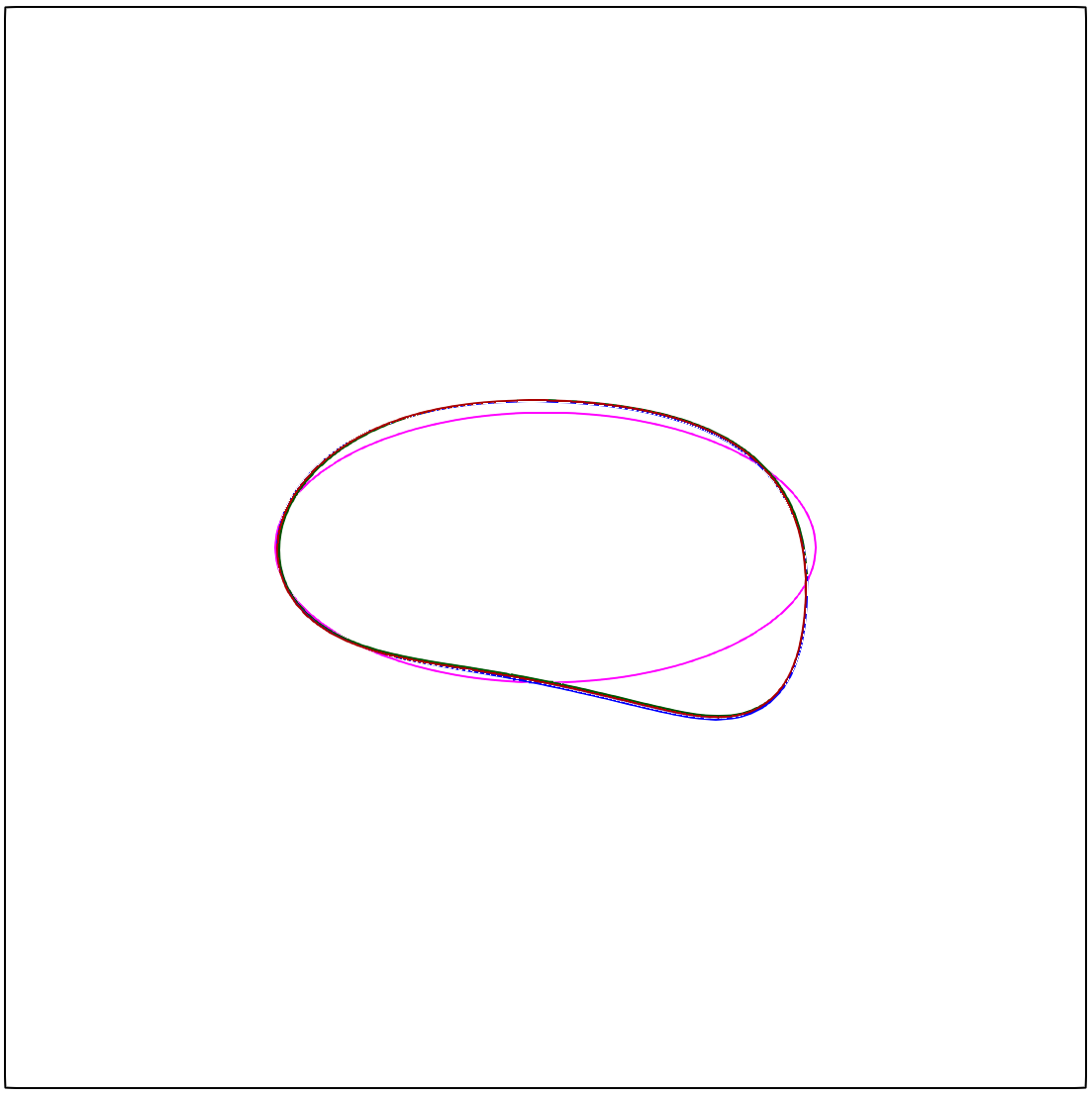} 
        \label{fig:Ex2-circle-50}}
               \hfill
              \subfloat[]{\includegraphics[width=0.45\textwidth]
    {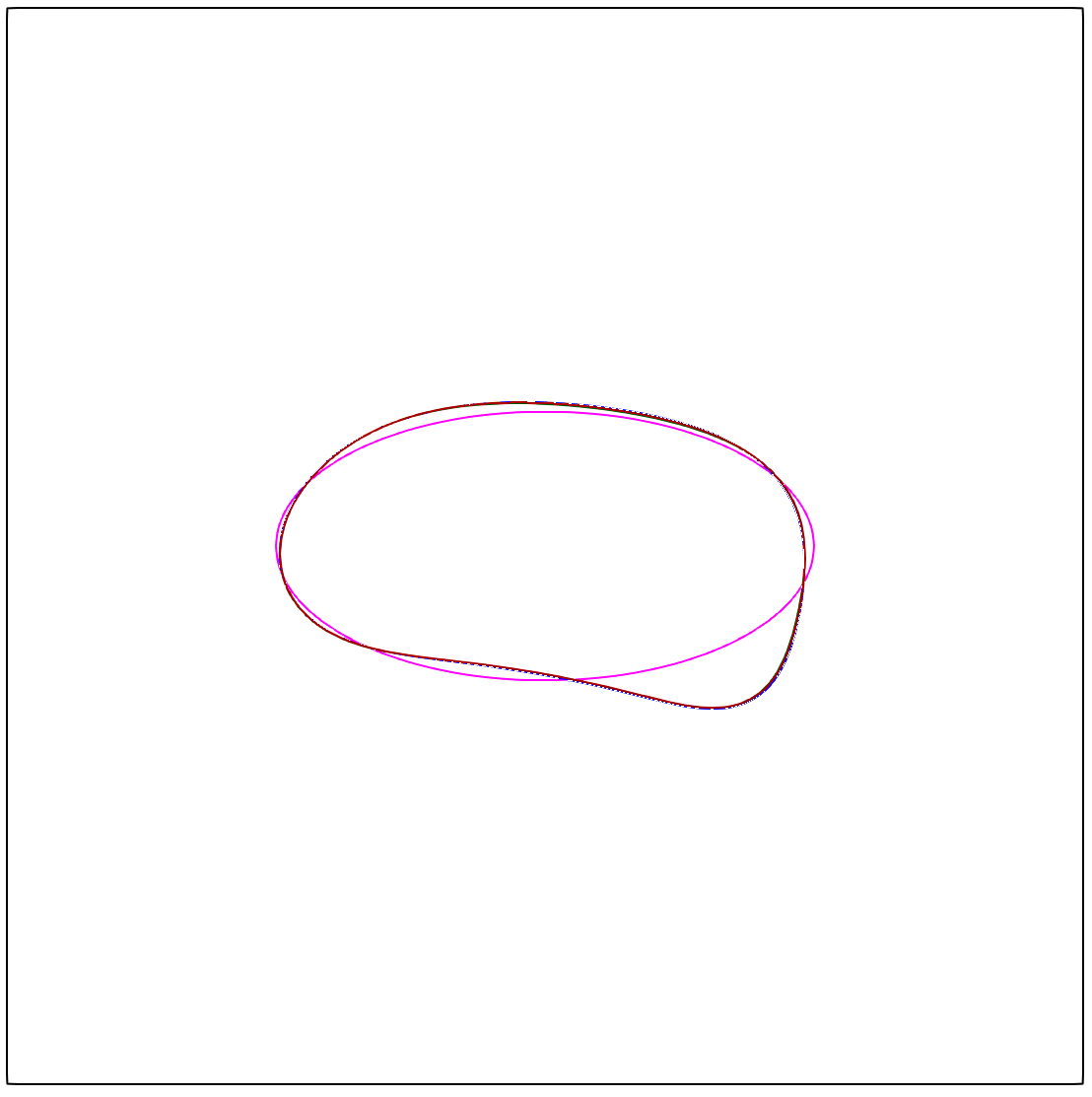} 
        \label{fig:Ex2-circle-120}}
               \hfill
        \subfloat[]{\includegraphics[width=0.45\textwidth]
    {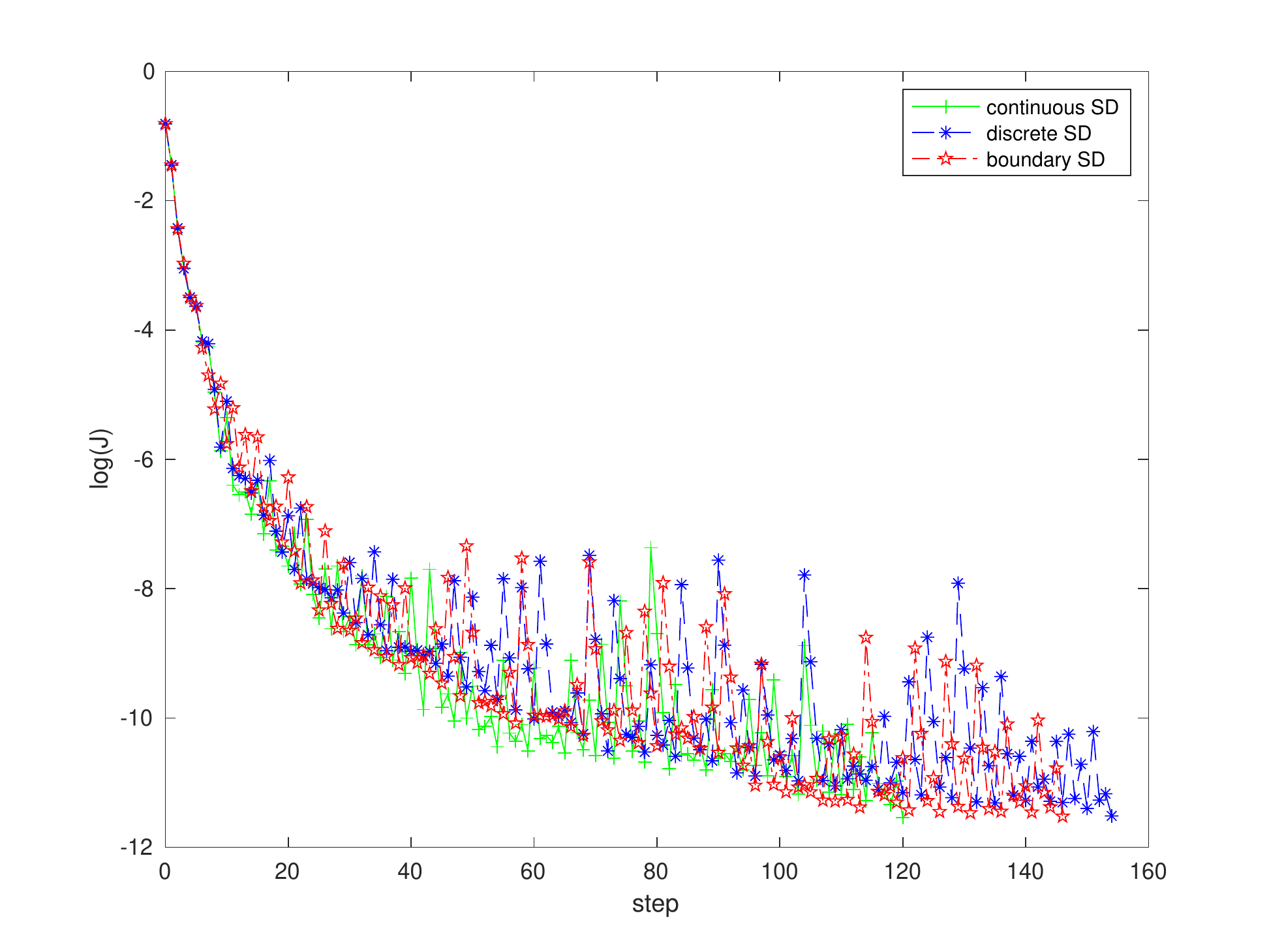} 
        \label{fig:Ex2-ellipse-circle-residual}}
\caption{\cref{ex:2}: level sets at steps $0${\upshape{(a)}}, $5${\upshape{(b)}}, $10${\upshape{(c)}}, $50${\upshape{(d)}} and $120${\upshape{(e)}} and the comparison of residual evolution {\upshape{(f)}}}
\end{figure}

\begin{example}[Lam\'e Square]\label{ex:lame-square}
For this example  the free boundary $\Gamma_\O$ is a  Lam\'e Square that has the following level set representation (see the magenta curve in \cref{fig:ex3-a0}):
\[
	\phi(x,y) = -81(x - 0.5)^n - 1296(y - 0.5)^n +1, \quad n=4 .
\]
The level set becomes closer to a rectangle as the integer $n$ increases.
We chose  the data such that $f=0$, $g_N = (5\sin(\theta),5 \cos(\theta)) \cdot \bn$ where $\theta =\tan^{-1}( (y-0.5)/(x-0.5))$ and $g_D$ is obtained by solving the forward problem on a $500 \times 500$ mesh.
\end{example}

We firstly start with circle as the initial level set, (see the red circle in \cref{fig:ex3-a0})
\[
	\phi(x,y) = -\sqrt{(x-0.5)^2 + (y - 0.5)^2} + 1/8.
\]
\cref{fig:ex3-a}-\cref{fig:ex3-d} show the level sets at steps $5, 10$ and $50$ and $150$
obtained by the continuous SD (green), discrete SD (blue) and boundary SD (red). With the stopping criteria that $J \le 5E-6$ and maximal iteration number not exceeds $200$, it takes $173$, $174$, and $200$ steps respectively using the continuous SD, discrete SD  and boundary SD. In this case, again, continuous and discrete SDs behaves almost identical. However, the level sets produced by the modified SD are slightly different.
\begin{figure}
    \centering
       \subfloat[]{\includegraphics[width=0.45\textwidth]
    {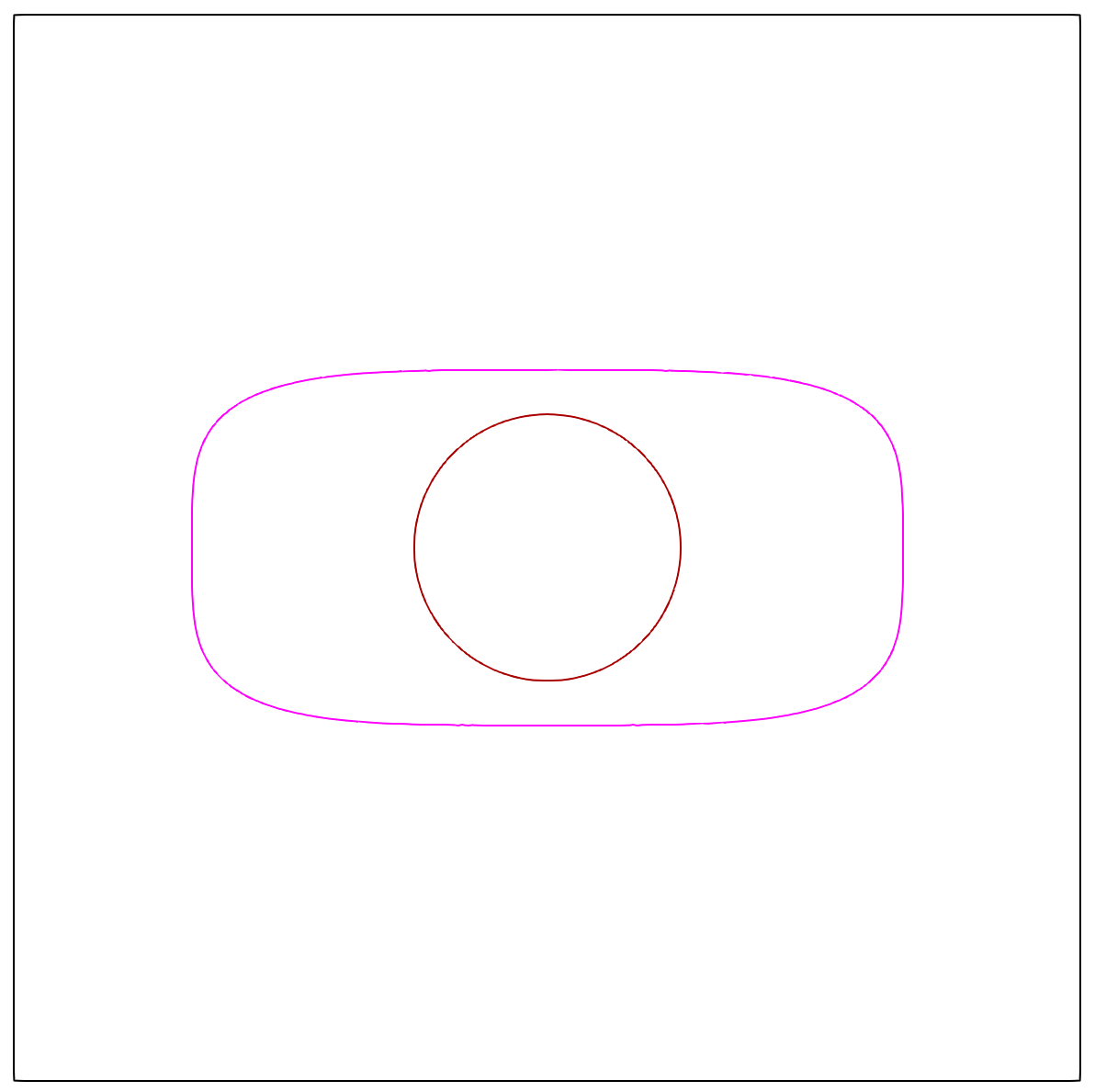} 
        \label{fig:ex3-a0}}
    \hfill
    \subfloat[]{\includegraphics[width=0.45\textwidth]
    {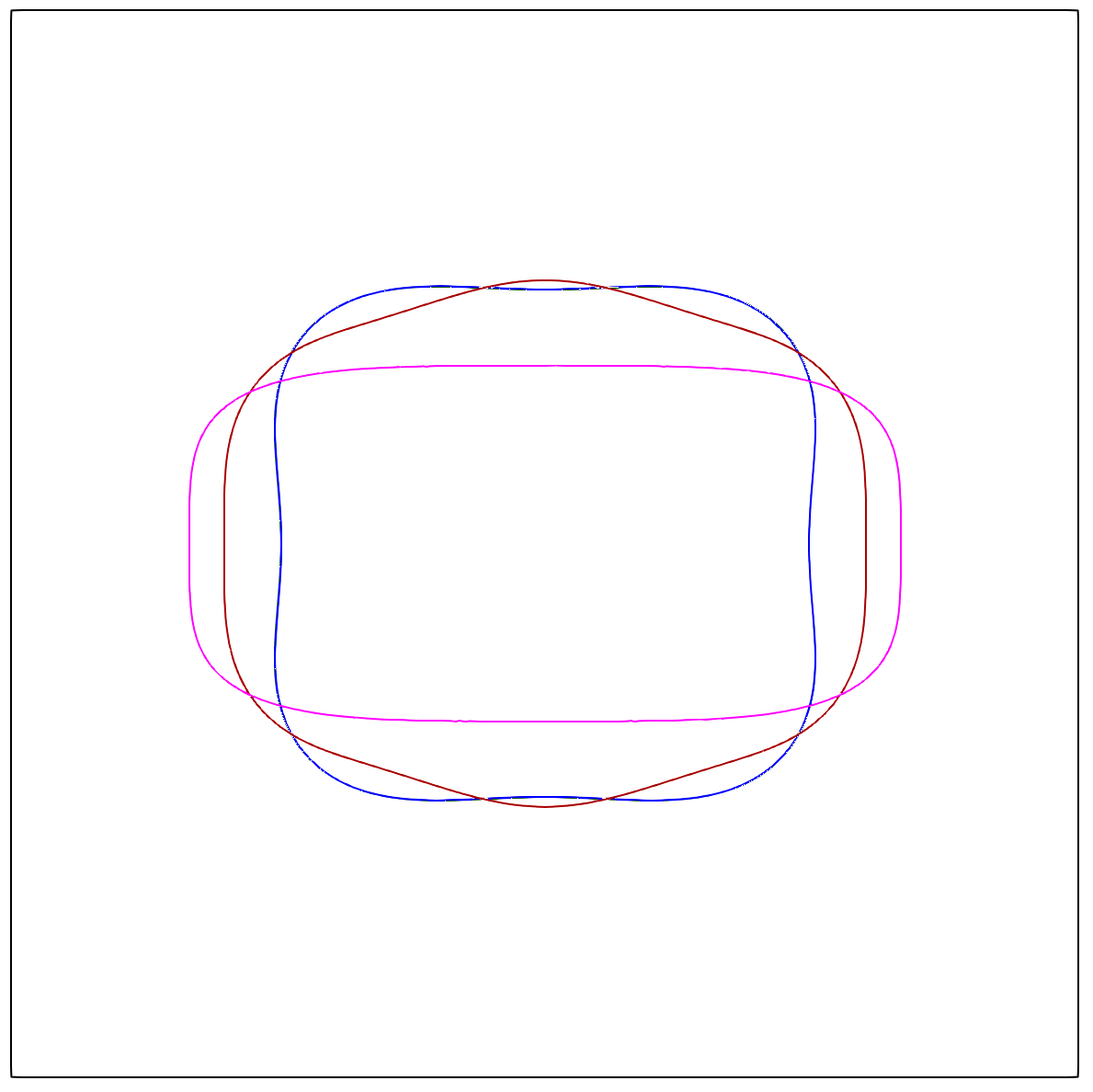} 
        \label{fig:ex3-a}}
    \hfill
        \subfloat[]{\includegraphics[width=0.45\textwidth]
    {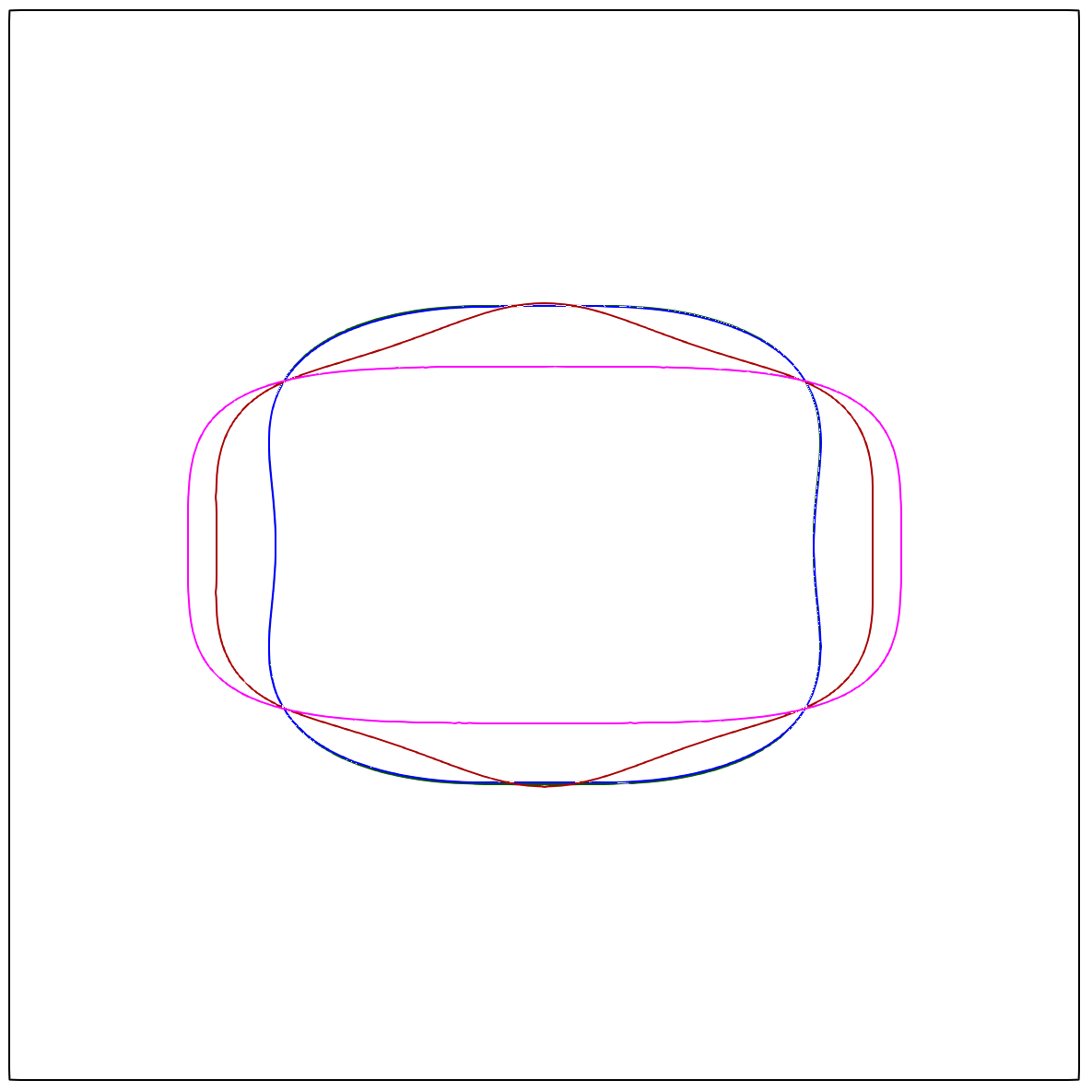} 
        \label{fig:ex3-a1}}
    \hfill
        \subfloat[]{\includegraphics[width=0.45\textwidth]{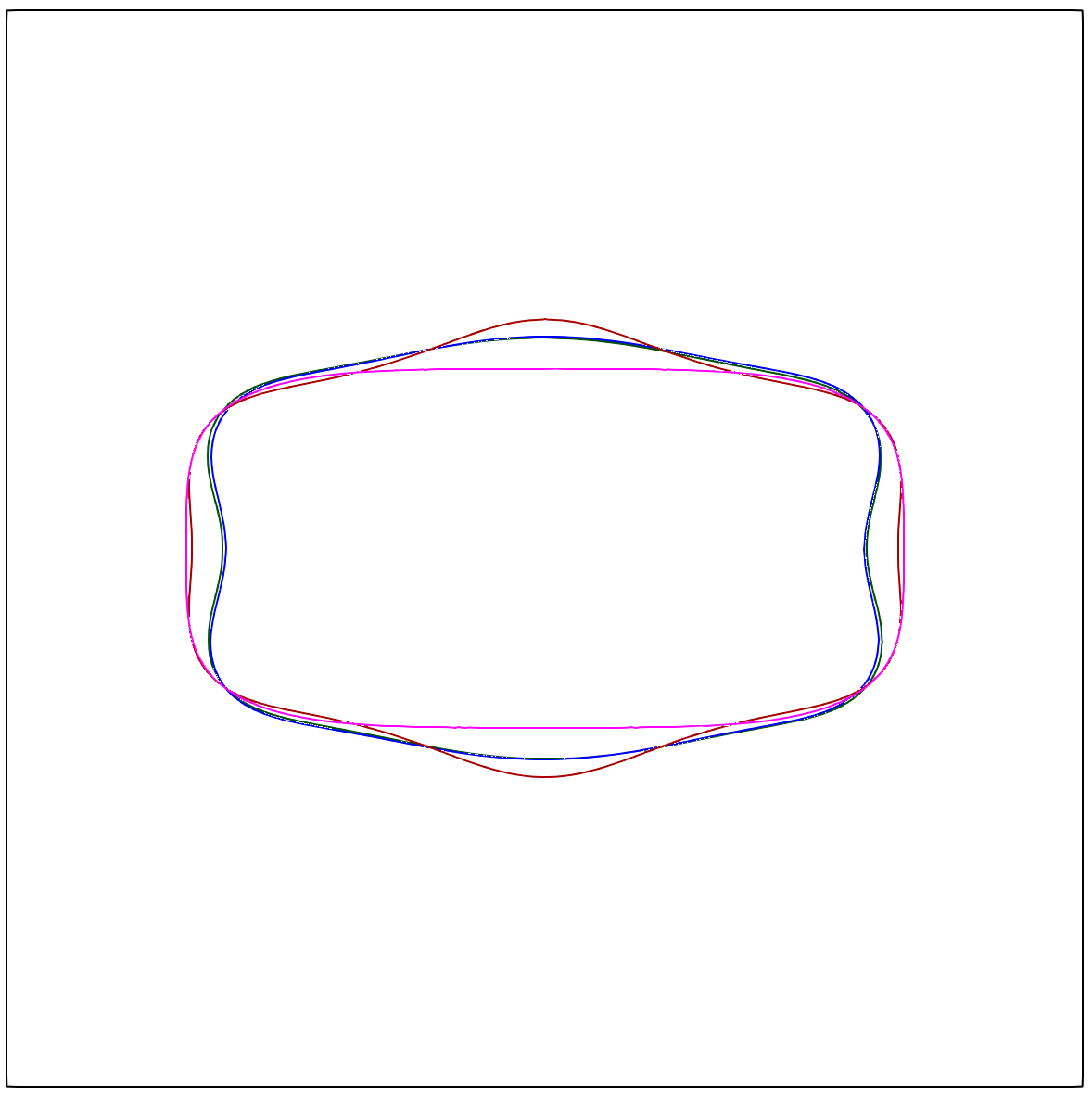}
       \label{fig:ex3-b}}
       \hfill
        \subfloat[]{\includegraphics[width=0.45\textwidth]
    {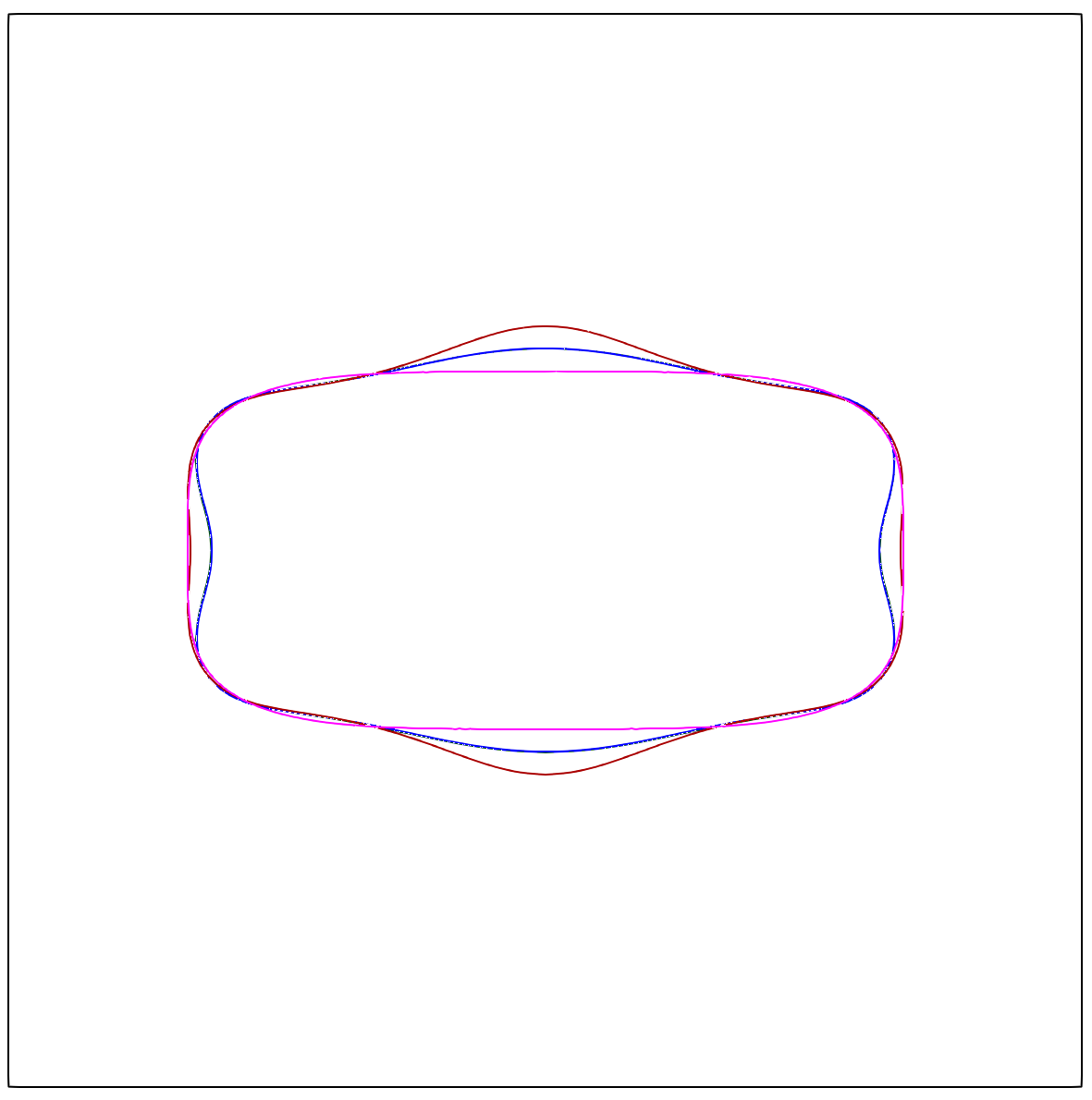} 
        \label{fig:ex3-c}}
              \hfill
        \subfloat[]{\includegraphics[width=0.45\textwidth]
    {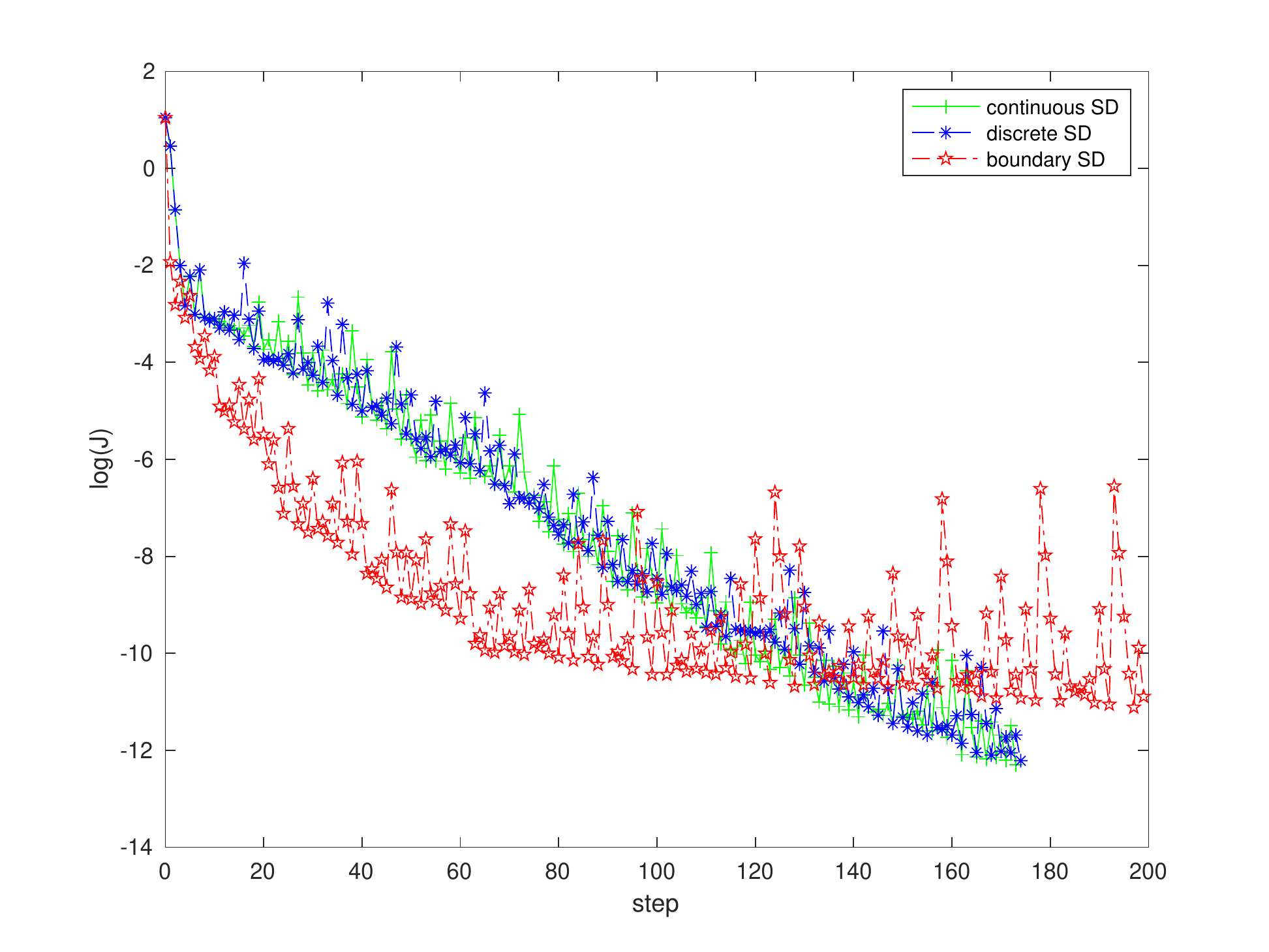} 
        \label{fig:ex3-d}}
\caption{\cref{ex:lame-square}: level sets at steps $0${\upshape{(a)}}, $5${\upshape{(b)}}, $10${\upshape{(c)}}, $50${\upshape{(d)}} and $150${\upshape{(e)}} and the comparison of residual evolution {\upshape{(f)}}}
\end{figure}


\begin{example}[Topology change of merging]
In this test, we aim to validate the ability of topology change for our algorithm.
We start with an initial guess of two separate Lam\'e squares with the following initial level set (see the red curves in \cref{fig:ex3-aa}):
\[
	\phi(x,y) = \max\left( \phi_1(x,y), \phi_2(x,y) \right),
\]
where
$\phi_1(x,y) = 1-1296(x - 0.32)^4 - 1296(y - 0.5)^4$ and $\phi_2(x,y)= 1-1296(x - 0.68)^4 - 1296(y - 0.5)^4$.
The stopping criteria is set such that $J \le 5E-6$.
It takes $271$, $276$, and $129$ steps for the respective continuous SD, discrete SD and boundary SD to reach the stopping criteria.
\cref{fig:ex3-aa} -\cref{fig:ex3-ff} show the level sets at the respective steps $0$, $10$, $50$, $100$ and the last step of level sets obtained by the continuous SD (green), discrete SD (blue) and boundary SD (red). 
We observe the topology change of merging in this case. 


\begin{figure}
    \centering
    \subfloat[]{\includegraphics[width=0.45\textwidth]{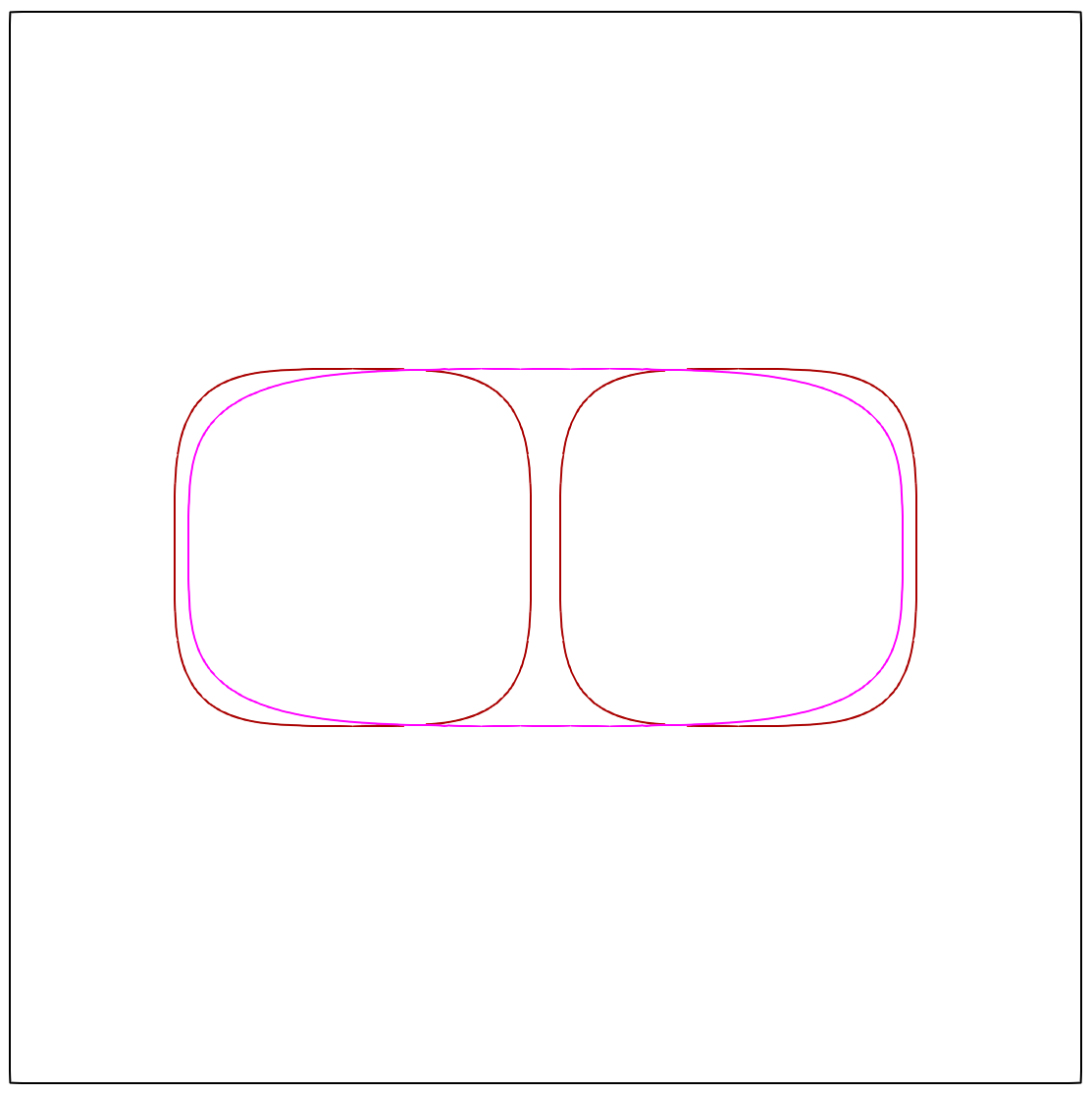} 
        \label{fig:ex3-aa}}
    \hfill
        \subfloat[]{\includegraphics[width=0.45\textwidth]{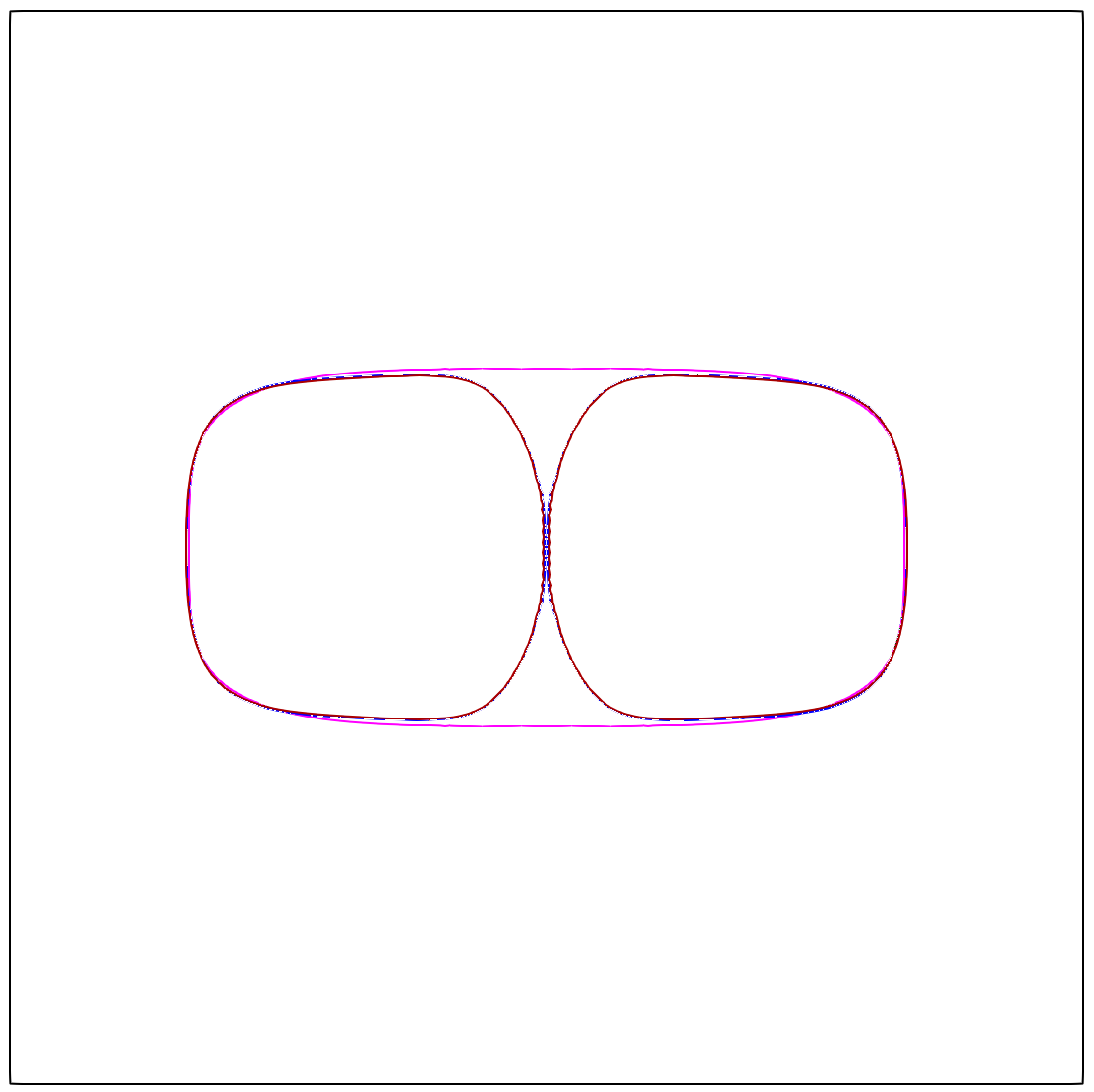}
       \label{fig:ex3-bb}}
       \hfill
        \subfloat[]{\includegraphics[width=0.45\textwidth]{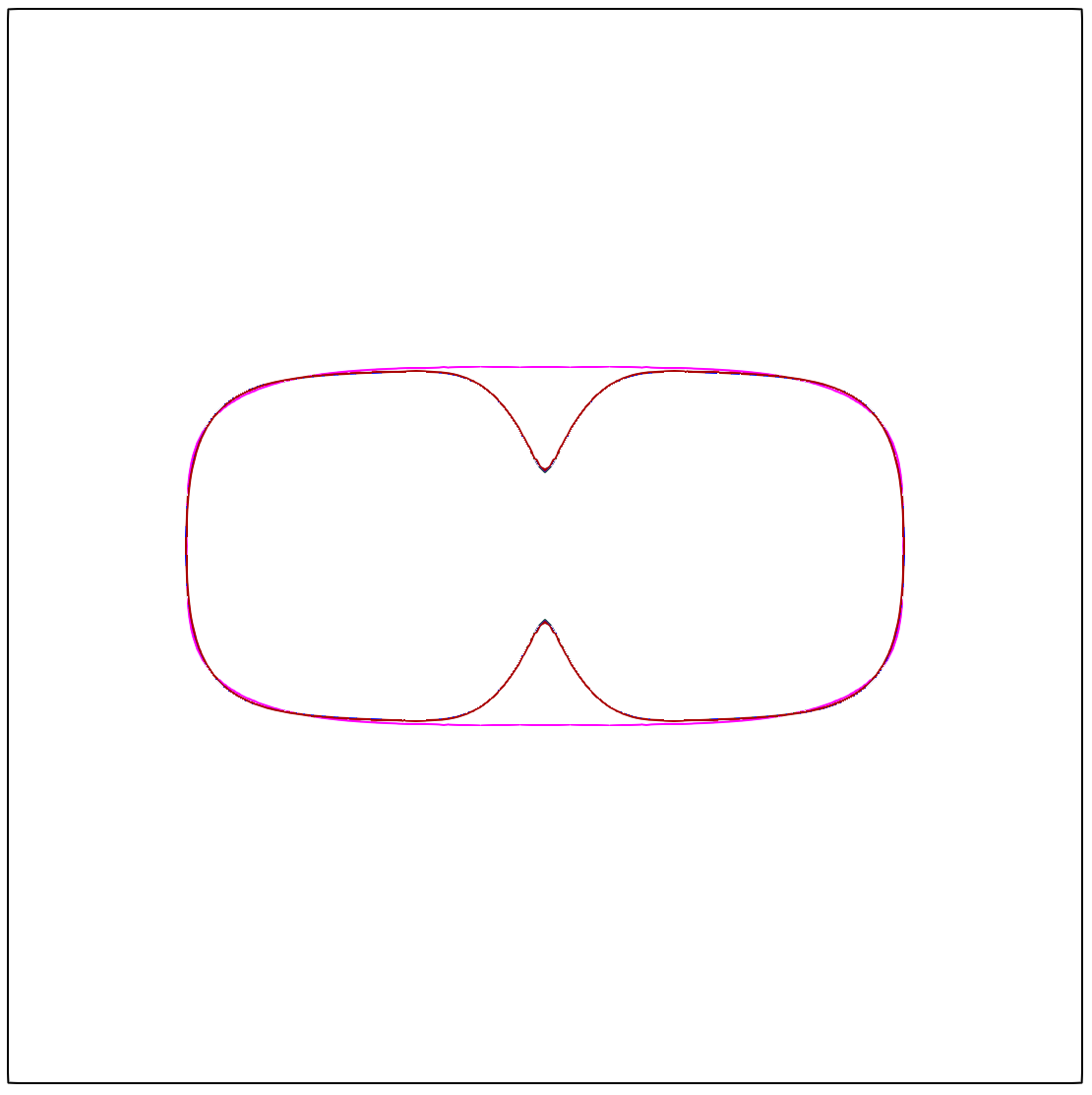} 
        \label{fig:ex3-cc}}
            \hfill
        \subfloat[]{\includegraphics[width=0.45\textwidth]{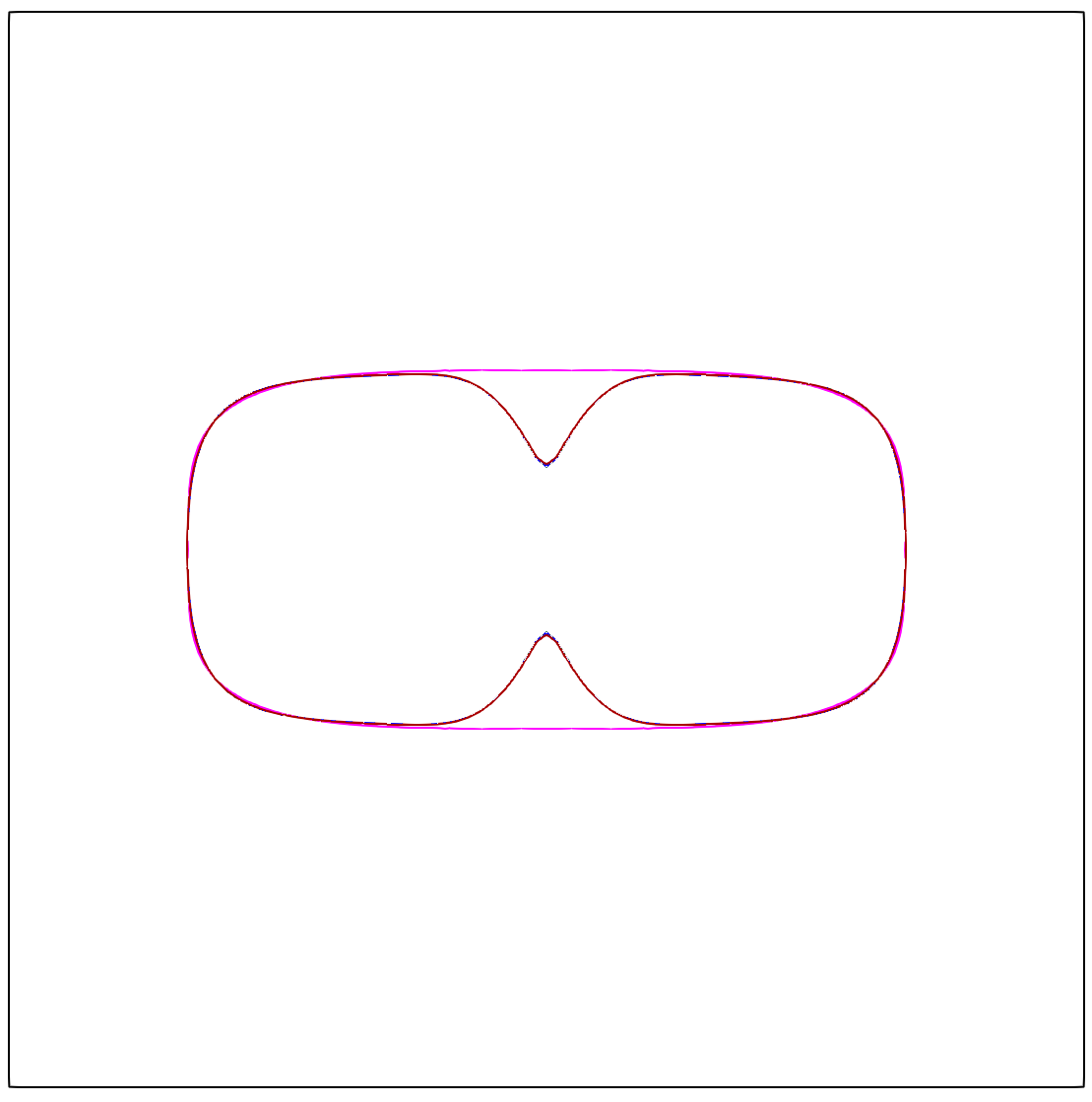}
       \label{fig:ex3-dd}}
                          \hfill
        \subfloat[]{\includegraphics[width=0.45\textwidth]{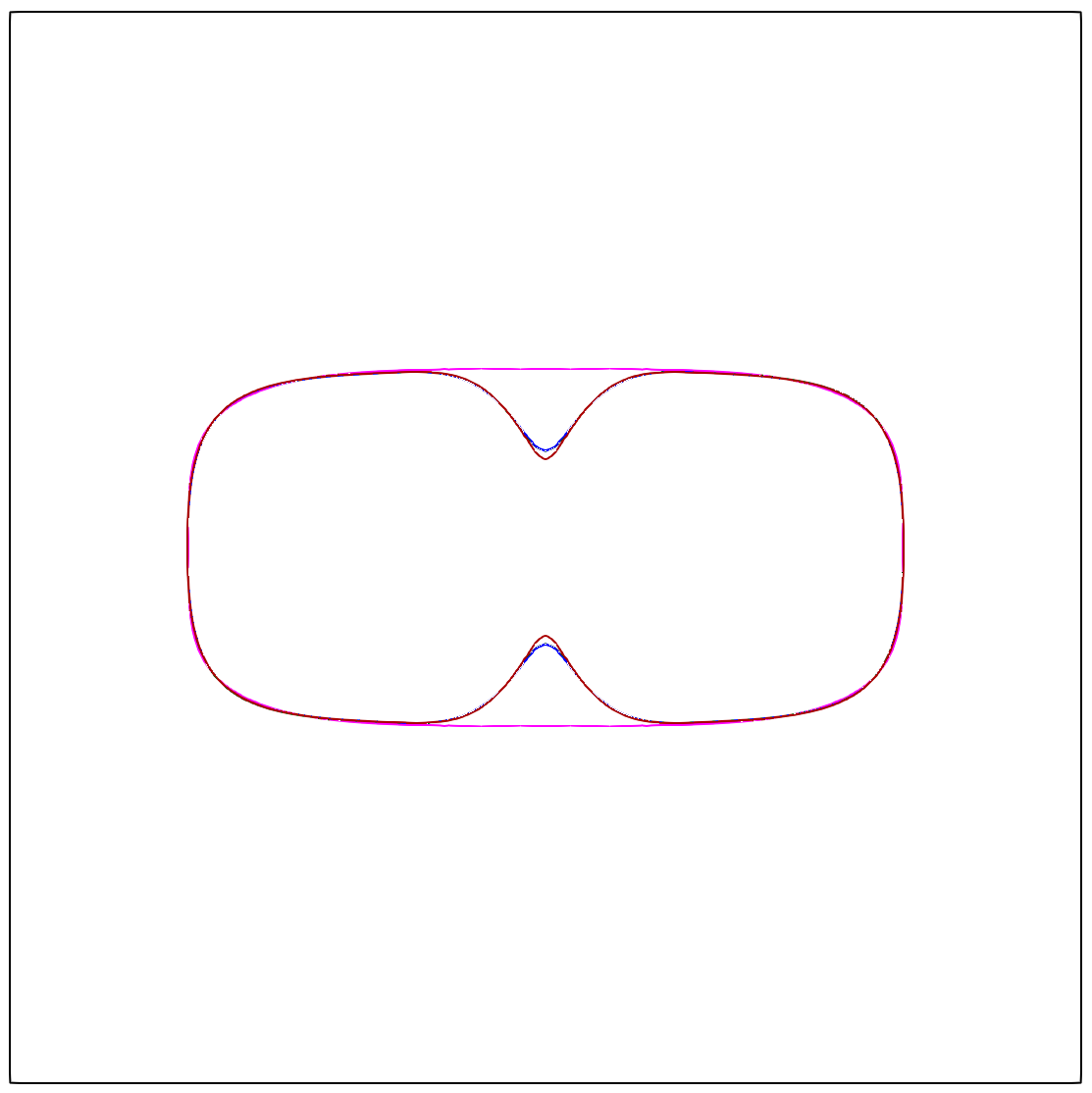}
       \label{fig:ex3-ff}}
                                \hfill
        \subfloat[]{\includegraphics[width=0.45\textwidth]{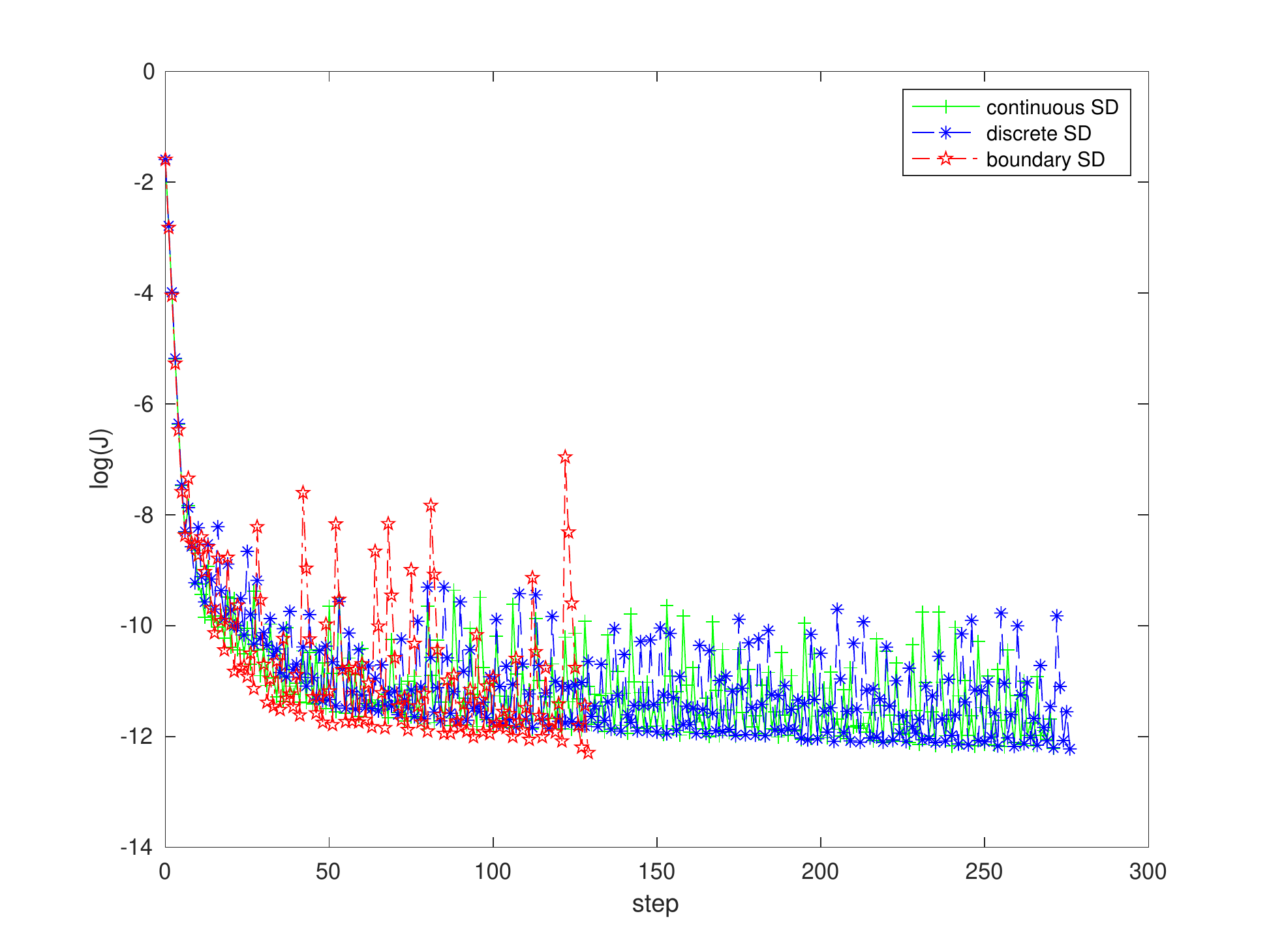}
       \label{fig:ex3-gg}}
\caption{\cref{ex:lame-square}: at steps $0${\upshape{(a)}}, $10${\upshape{(b)}}, $50${\upshape{(c)}}, 
$100${\upshape{(d)}}, the final {\upshape{(e)}}, and the comparison of residual evolution  {\upshape{(f)}}  .}
\end{figure}

\end{example}

\begin{example}[Doubly connected domain]\label{Ex4:double-domain}
In this example, we aim to identify the following level set with two isolated circles (see the magenta curve in  \cref{fig:ex5-a}):
\[
	\phi(x,y) = \max\left(0.15 - \sqrt{ (x - 0.2)^2 + (y -0.5 )^2}, 0.15 - \sqrt{ (x - 0.80)^2 + (y -0.5 )^2} \right).
\]
We firstly test with a connected Cassini oval (see the red curve in \cref{fig:ex5-a}):
\[
	\phi(x,y) = -(\hat x^2 + \hat y^2)^2 + 2(\hat x^2 - \hat y^2) - 1 + b^4, 
	\quad \hat x = 3x - 1.5, \quad  \hat y = 3y - 1.5,  \quad b = 1.001.
\]
The stopping criteria is set such that the maximal number of iteration not exceeds $300$. 
We chose  the data such that $f=0$, $g_N = (x - 0.5, y-0.5) \cdot \bn$  and $g_D$ is obtained by solving the forward problem on a $500 \times 500$ mesh.
Figures \cref{fig:ex5-b}--\ref{fig:ex5-d} show the level sets at the respective steps $50,100, 200$ and $300$. This example validates the capability of the algorithm in the topology change of splitting. During the process, the Cassini oval initially splits into two cone-like shapes and then each gradually evolve into a circle. The convergence is, however, quite slow and it is likely due to the sharp angles evolved after splitting. The results generated by the three SDs are again quite similar.

\begin{figure}
    \centering
    \subfloat[]{\includegraphics[width=0.45\textwidth]{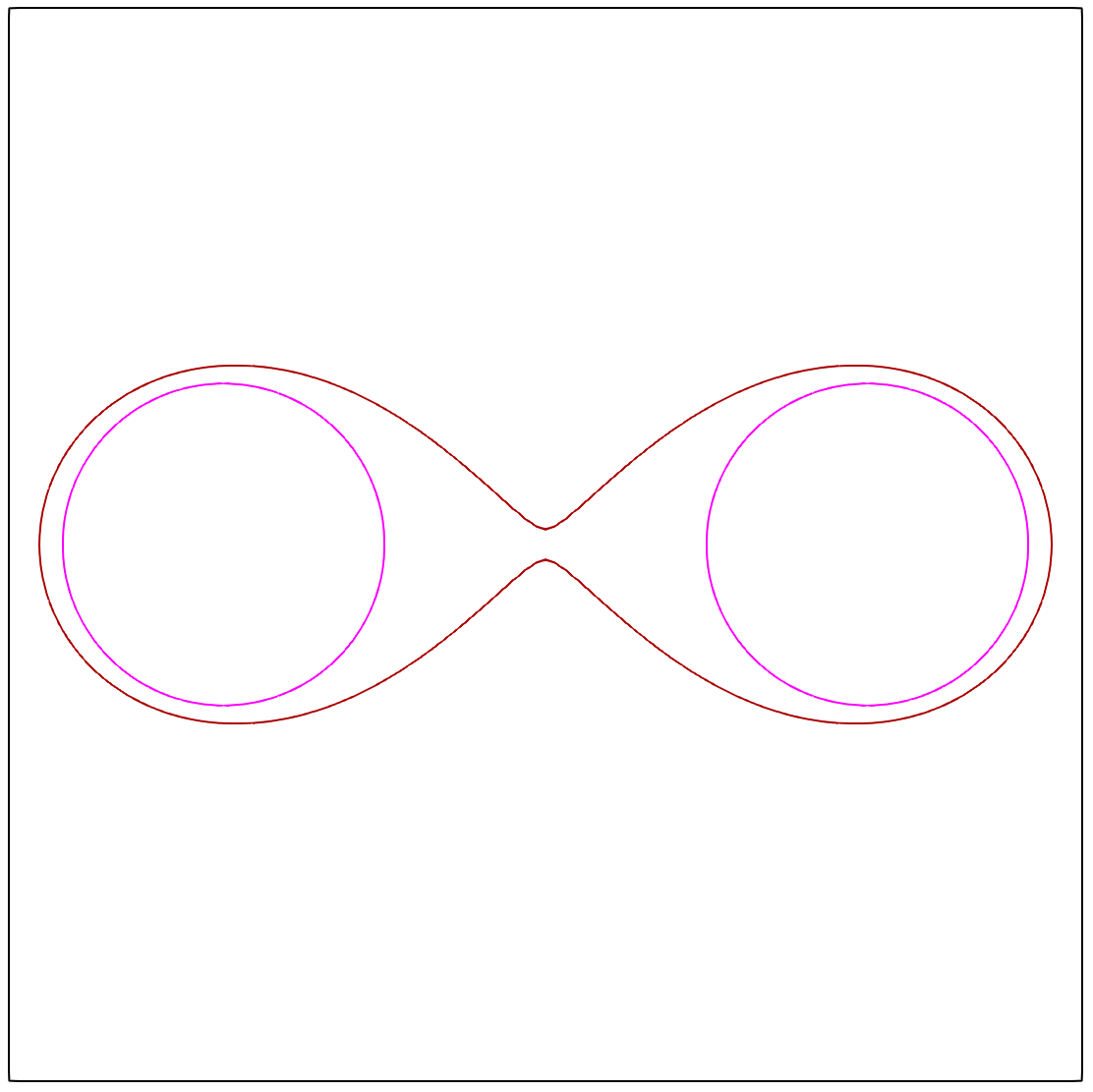} 
        \label{fig:ex5-a}}
    \hfill
        \subfloat[]{\includegraphics[width=0.45\textwidth]{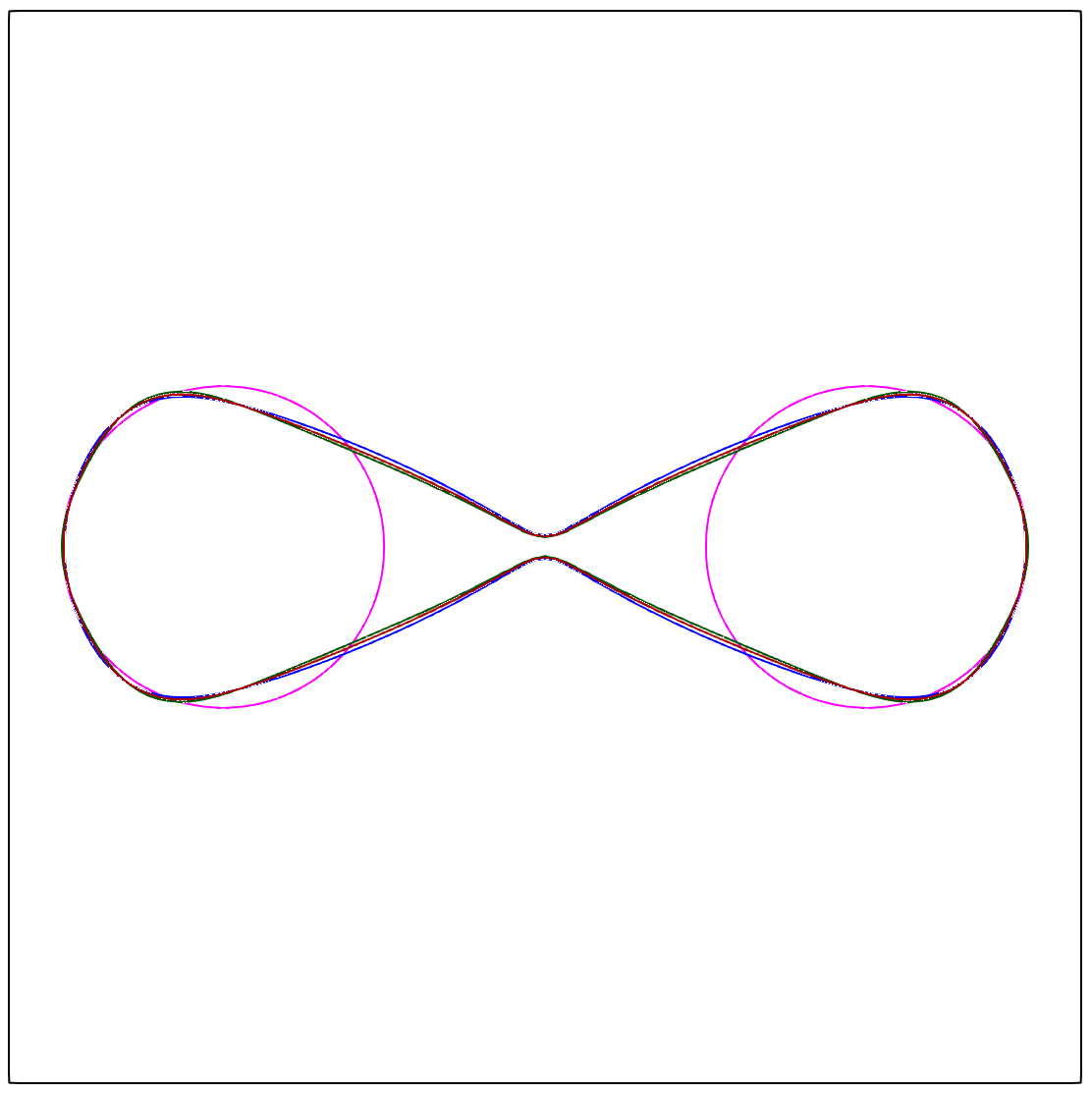}
       \label{fig:ex5-b}}
       \hfill
        \subfloat[]{\includegraphics[width=0.45\textwidth]{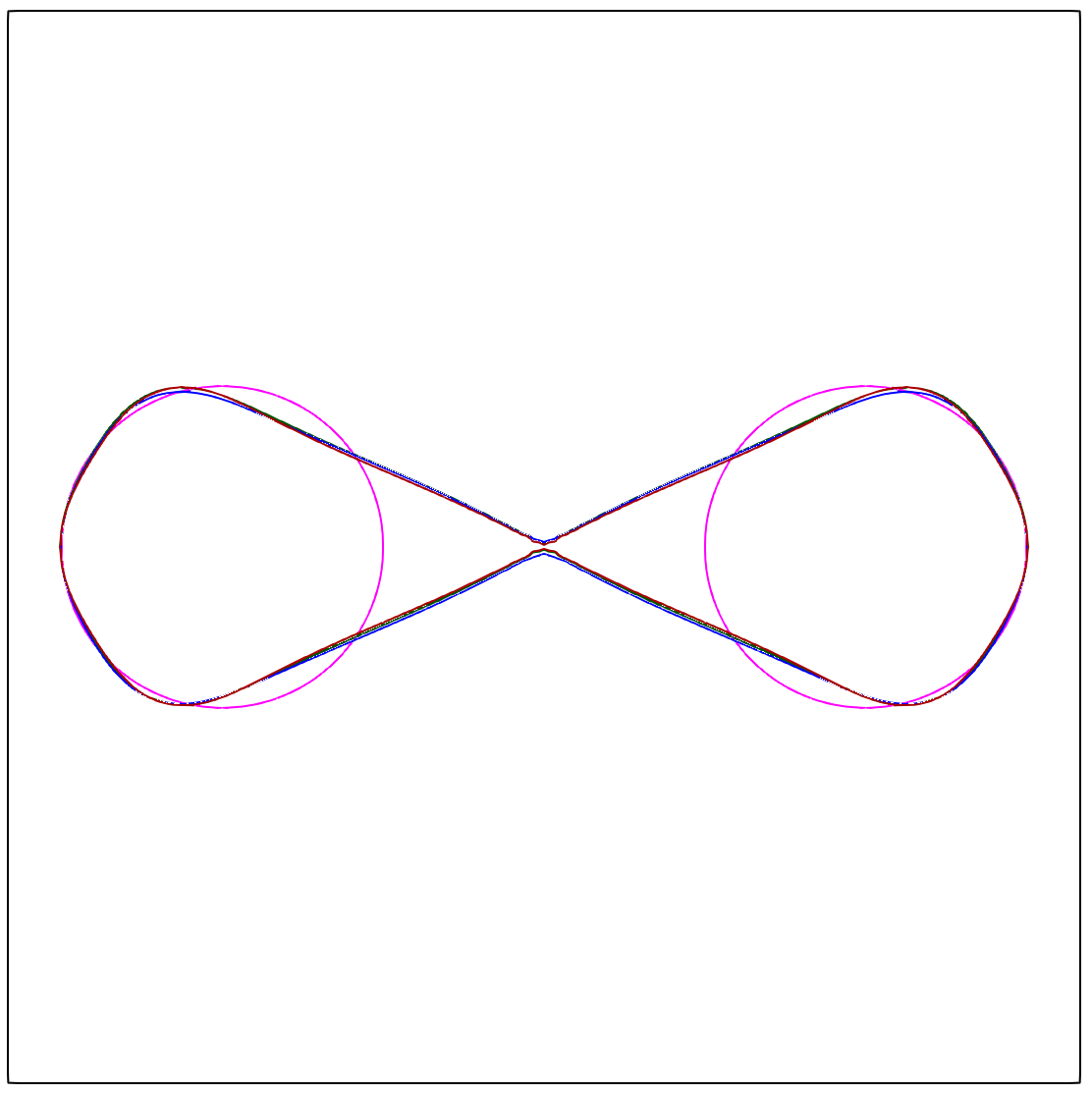} 
        \label{fig:ex5-c}}
        \hfill
        \subfloat[]{\includegraphics[width=0.45\textwidth]{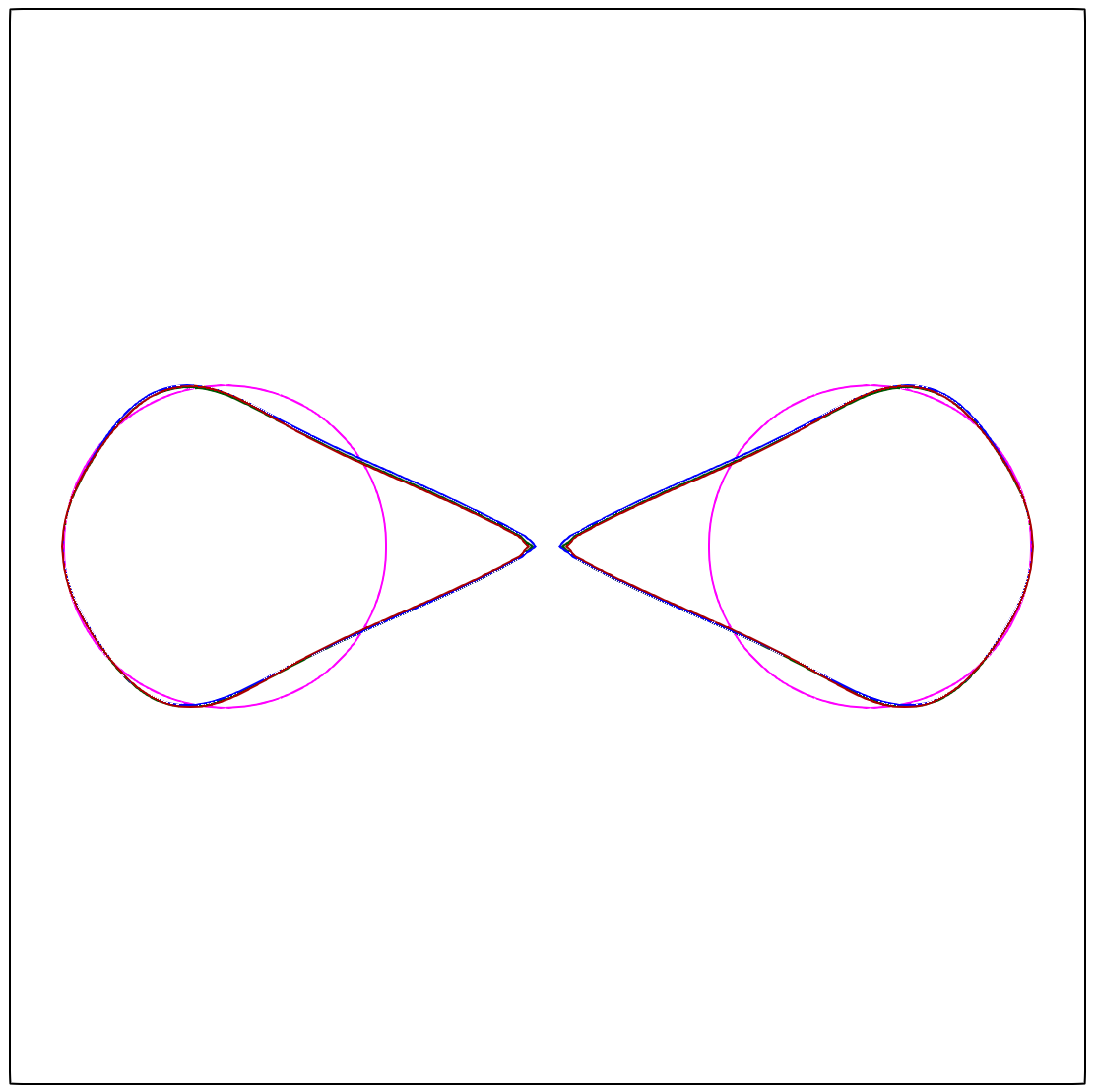} 
        \label{fig:ex5-d}}
                \hfill
        \subfloat[]{\includegraphics[width=0.45\textwidth]{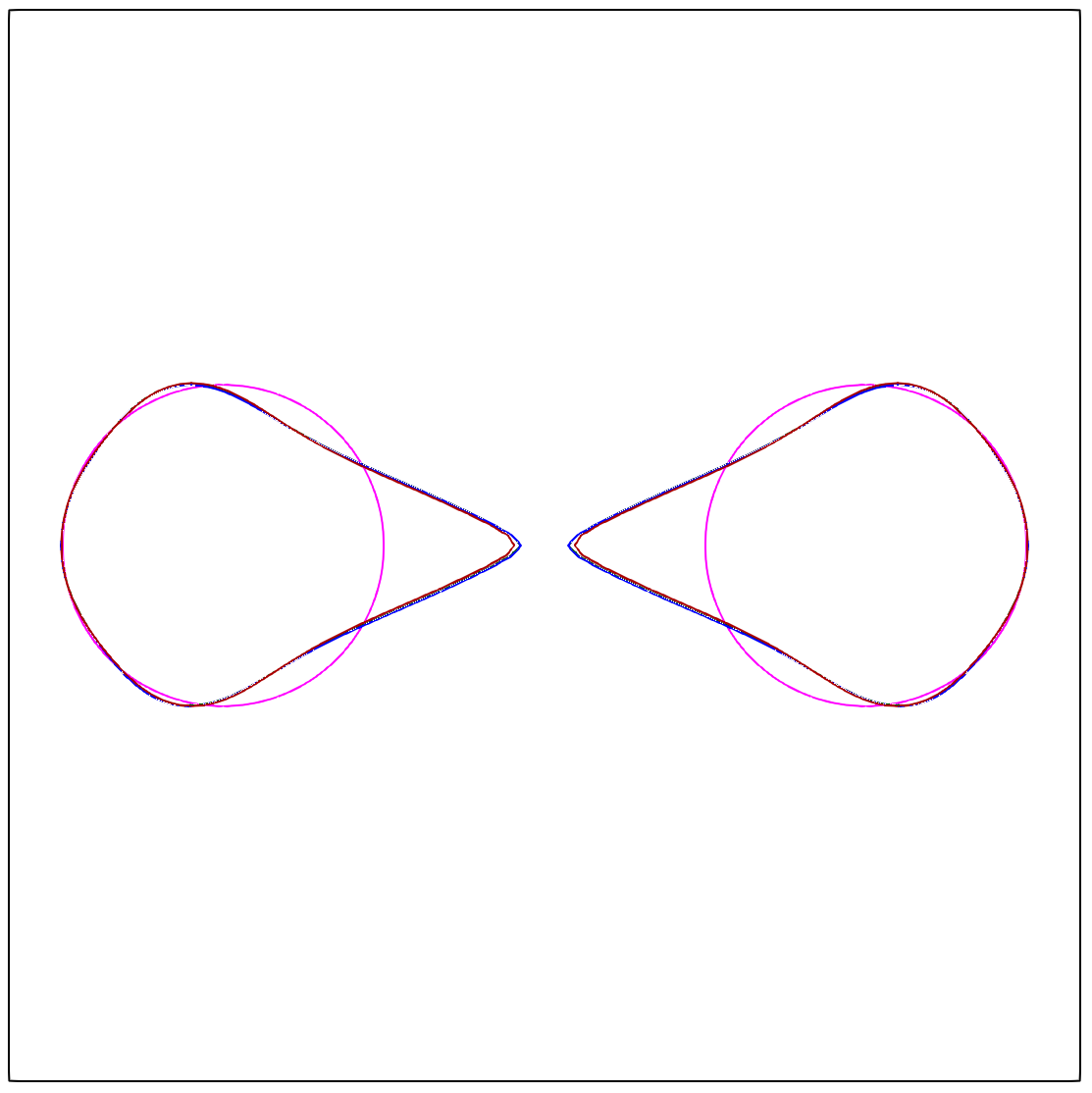} 
        \label{fig:ex5-e}}
                        \hfill
        \subfloat[]{\includegraphics[width=0.45\textwidth]{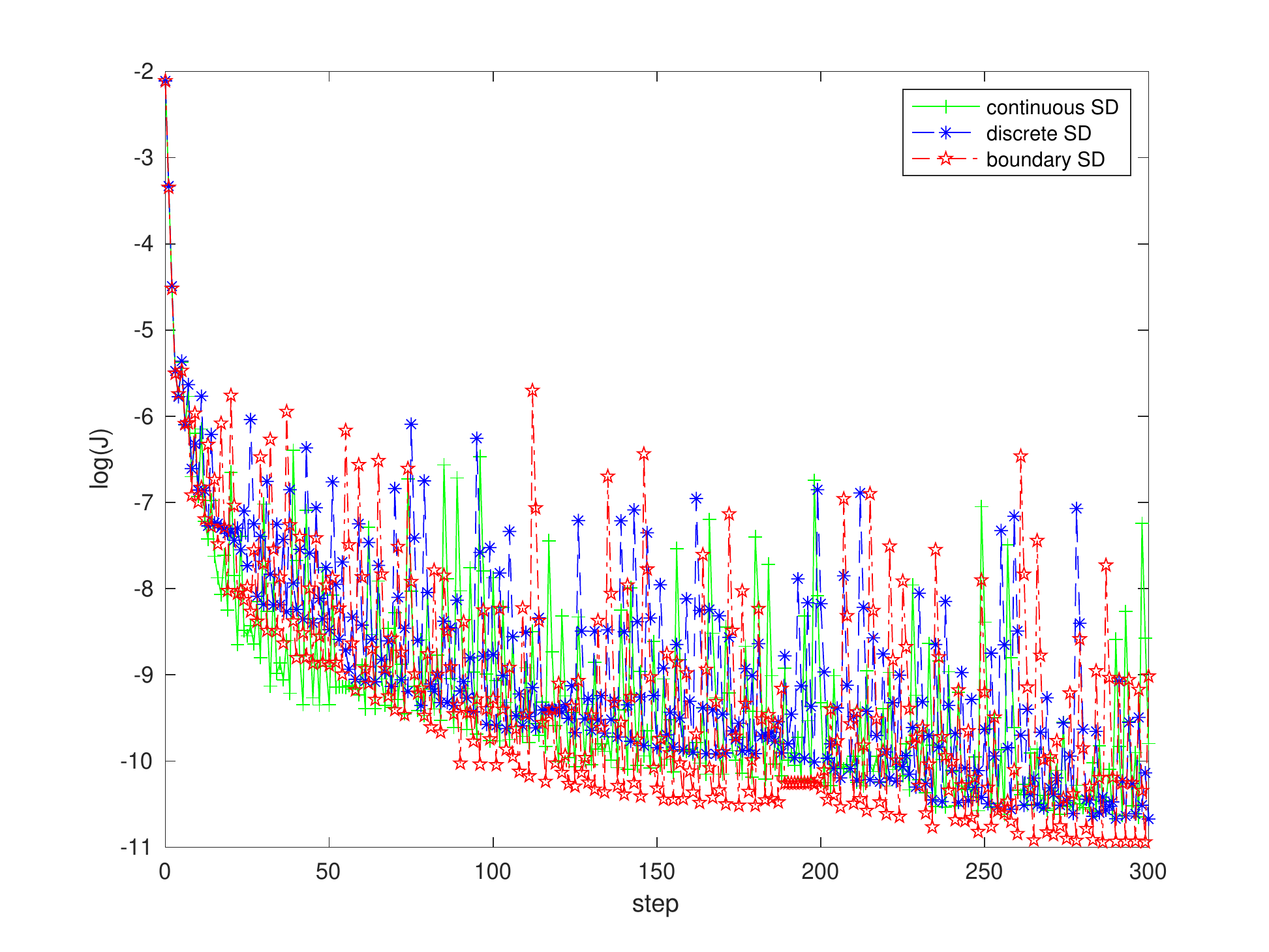} 
        \label{fig:ex5-f}}
\caption{\cref{Ex4:double-domain}: at steps $0${\upshape{(a)}}, $50${\upshape{(b)}}, $100${\upshape{(c)}}, $200${\upshape{(d)}} and $300${\upshape{(e)}} and the residual evolution}
\end{figure}

\end{example}
\section{Appendix}
Proof of \cref{lemma4.4}
\begin{proof}
By the assumption that $T_t$ is smooth, using similar arguments as in \cref{lem:shape-derivative} and \cref{lem:shape-derivative1} gives
\begin{equation}
\begin{split}
	&\int_{F^t} \jump{\nabla w \cdot \bn_t} \jump{\nabla  v\cdot \bn_t} \,ds\\
	&\quad = \int_{F} \jump{\nabla w \circ T^t \cdot (\bn_t \circ T_t)}\jump{\nabla v \circ T^t\cdot (\bn_t \circ T_t)}\omega(t)\,ds\\
	&\quad = \int_{F} \jump{A(t)\nabla (w \circ T^t) \cdot \bn}\jump{A(t)\nabla (v \circ T^t) \cdot \bn} \omega^{-1}(t)\,ds\\
\end{split}
\end{equation}
Applying product rule, we then have that
\begin{equation}
\begin{split}
	&D_{\O, \bftheta}\int_{F} \jump{\nabla w \cdot \bn} \jump{\nabla  v\cdot \bn} \,ds\\
	&=\int_F \jump{A'(0) \nabla w \cdot \bn} \jump{\nabla v \cdot \bn}
	+
	 \jump{A'(0) \nabla v \cdot \bn} \jump{\nabla w \cdot \bn}
	 \\
	&\quad + \int_{F} \jump{\nabla w \cdot \bn} \jump{\nabla  \dot v\cdot \bn}
	 + \jump{\nabla v \cdot \bn} \jump{\nabla  \dot w\cdot \bn} \,ds
	 \\
	 &\quad - \int_{F} \jump{\nabla w \cdot \bn} \jump{\nabla  v\cdot \bn} \, \omega'(0)ds 
	\\
	&= \int_{F} \jump{ (\nabla \cdot \bftheta) \nabla w \cdot \bn - S(\theta) \cdot \nabla w \cdot \bn} \jump{\nabla  v\cdot \bn} \,ds + \int_{F} \jump{\nabla \dot w \cdot \bn} \jump{\nabla  v\cdot \bn} \,ds
	\\
	&\quad + \int_{F} \jump{ (\nabla \cdot \bftheta) \nabla v \cdot \bn - S(\theta) \cdot \nabla v \cdot \bn} \jump{\nabla  w\cdot \bn} \,ds + \int_{F} \jump{\nabla w \cdot \bn} \jump{\nabla  \dot v\cdot \bn} \,ds
	\\
	&\quad  - \int_{F} \jump{\nabla w \cdot \bn} \jump{\nabla  v\cdot \bn} \, (\divvr \bftheta - (D\bftheta \cdot \bfn)  \cdot \bfn)ds.
\end{split}
\end{equation}
This completes the proof of  \cref{lem:shape-derivative4}.
\end{proof}

\bibliographystyle{siamplain}
\bibliography{references}
\end{document}